\pdfoutput=1
\documentclass[colorlinks,10pt,reqno]{amsart}

\usepackage{verbatim}
\usepackage{times}
\usepackage{stmaryrd}
\usepackage[a4paper]{geometry}
\usepackage{tabls}
\usepackage{url}
\usepackage[linktocpage = true]{hyperref}
\hypersetup{
 colorlinks = true,
 citecolor = blue,
}
\usepackage{color}
\definecolor{red}{rgb}{1,0,0}
\definecolor{green}{rgb}{0,1,0}
\definecolor{blue}{rgb}{0,0,1}
\definecolor{refkey}{gray}{.625}
\definecolor{labelkey}{gray}{.625}

\usepackage[lite]{amsrefs}
\usepackage{amsfonts}
\usepackage{amssymb}
\usepackage{mathrsfs}
\usepackage{amsmath}
\usepackage{amsthm}
\usepackage{amscd}
\usepackage{enumerate}
\usepackage{paralist}
\usepackage{graphicx}
\usepackage{mathtools}
\usepackage{kantlipsum}
\usepackage{tikz-cd}
\allowdisplaybreaks

\DeclareMathOperator{\id}{id}
\DeclareMathOperator{\pr}{pr}
\DeclareMathOperator{\Der}{Der}
\DeclareMathOperator{\Bott}{Bott}

\DeclareMathOperator{\Td}{Td}
\DeclareMathOperator{\tr}{tr}
\DeclareMathOperator{\str}{str}

\makeatletter
 \def\title@font{\normalsize\bfseries}
 \let\ltx@maketitle\@maketitle
 \def\@maketitle{\bgroup
 \let\ltx@title\@title
 \def\@\title{\resizebox{\textwidth}{!}{
  \mbox{\title@font\ltx@title}
 }}
 \ltx@maketitle
 \egroup}
\makeatother

\newcommand{\abs}[1]{\lvert#1\rvert}

\newcommand{\T}{\mathcal{T}}

\newcommand{\R}{\mathbb{R}}

\newcommand{\At}{\operatorname{At}}

\numberwithin{equation}{section}
\theoremstyle{plain}
\newtheorem{corollary}[equation]{Corollary}
\newtheorem{theorem}[equation]{Theorem}
\newtheorem{lemma}[equation]{Lemma}
\newtheorem{proposition}[equation]{Proposition}

\theoremstyle{definition}
\newtheorem{definition}[equation]{Definition}
\newtheorem{example}[equation]{Example}
\theoremstyle{remark}
\newtheorem{remark}[equation]{Remark}

\begin{document}
\def\C{\mathbb{C}}
\def\CE{\mathrm{CE}}
\def\D{\mathcal{D}}
\def\E{\mathcal{E}}
\def\F{\mathcal{F}}
\def\G{\mathcal{G}}
\def\M{\mathcal{M}}
\def\N{\mathcal{N}}
\def\O{\mathcal{O}}
\def\Q{\mathcal{Q}}
\def\P{\mathcal{P}}
\def\T{\mathcal{T}}
\def\U{\mathcal{U}}
\def\V{\mathscr{V}}
\def\L{\mathcal{L}}
\def\W{\mathcal{W}}
\def\X{\mathcal{X}}
\def\Y{\mathcal{Y}}
\def\Im{\operatorname{Im}}
\def\End{\operatorname{End}}
\def\Ber{\operatorname{Ber}}
\def\Hom{\operatorname{Hom}}
\def\Bott{\operatorname{Bott}}
\def\sgn{\operatorname{sgn}}
\def\rk{\operatorname{rank}}
\def\ch{\operatorname{ch}}
\newcommand{\bas}{\mathrm{bas}}
\def\Mod{\operatorname{Mod}}

\title{Atiyah and Todd classes of regular Lie algebroids}

\author{Maosong Xiang}
\address{School of Mathematics and Statistics, Center for Mathematical Sciences, Huazhong University of Science and Technology}
\email{\href{mailto: msxiang@hust.edu.cn}{msxiang@hust.edu.cn}}
\thanks{Research partially supported by NSFC grant 11901221}

\begin{abstract}
  For any regular Lie algebroid $A$, the kernel $K$ and the image $F$ of its anchor map $\rho_A$, together with $A$ itself fit into a short exact sequence, called Atiyah sequence, of Lie algebroids.
  We prove that Atiyah and Todd classes of dg manifolds arising from regular Lie algebroids respect the Atiyah sequence. That is, the Atiyah and Todd classes of $A$ restrict to the Atiyah and Todd classes of the bundle $K$ of Lie algebras on the one hand, and project onto the Atiyah and Todd classes of the integrable distribution $F \subseteq T_M$ on the other hand.
\end{abstract}

\maketitle

\\

\tableofcontents

\section{Introduction}
The Atiyah class was introduced by Atiyah~\cite{Atiyah} to characterize the obstruction to the existence of a holomorphic connection on a holomorphic vector bundle.  Kapranov~\cite{Kap} showed that the Atiyah class of a K\"{a}hler manifold $X$ induces an $L_\infty$ algebra structure on the shifted tangent complex $\Omega^{0,\bullet-1}_X(T_X)$, which plays an important role in his  reformulation of Rozansky-Witten theory. In his celebrating work~\cite{Kon}, Kontsevich found a deep link between the Todd genus of complex manifolds and the Duflo element of Lie algebras.
Liao, Sti\'{e}non and Xu explained in~\cite{LSX} this link via the formality theorem for smooth differential graded (dg for short) manifolds, where the Atiyah class of a dg vector bundle introduced by Mehta, Sti\'{e}non and Xu in~\cite{MSX} is an essential ingredient. See the survey article~\cite{SXsurvey} on Atiyah classes of dg manifolds and how they work in the Duflo-Kontsevich type theorems for dg manifolds.

A dg manifold (or a $Q$-manifold) is a $\mathbb{Z}$-graded smooth manifold equipped with a homological vector field $Q$, i.e., a degree $+1$ derivation of square zero on the algebra of functions. Dg manifolds arise naturally in many situations in geometry, Lie theory, and mathematical physics.
For example, according to Va\u{i}ntrob~\cite{Vai}, to any Lie algebroid $A$ is associated a dg manifold $(A[1], d_A)$, where the homological vector field $d_A$ is identified with Chevalley-Eilenberg differential of the Lie algebroid $A$. In fact, any homological vector field on the graded manifold $A[1]$ arises from a Lie algebroid structure on the vector bundle $A$ in this manner.
Since graded manifolds enjoy many similar properties as smooth manifolds, such as the existence of connections, curvatures, etc.,  one is able to generalize the same constructions in topology and differential geometry, for instance, characteristic classes by the Chern-Weil theory, to dg manifolds (cf.~\cites{CM2019, KS}).

In this paper, we study the Atiyah and Todd classes of dg manifolds arising from regular Lie algebroids.
A Lie algebroid $(A,\rho_A, [-,-]_A)$ over a smooth manifold $M$ is said to be regular if its anchor $\rho_A$ is of constant rank. The kernel $K = \ker(\rho_A)$ together with the restriction $[-,-]_K$ of the Lie bracket $[-,-]_A$ onto $\Gamma(K)$ is a bundle of Lie algebras; the image $F = \operatorname{Im}(\rho_A) \subseteq T_M$, being as the tangent bundle of the regular characteristic foliation, is a Lie subalgebroid of the tangent Lie algebroid $T_M$.
In other words, there is a short exact sequence of Lie algebroids over $M$
\[
  0 \rightarrow K \xrightarrow{i} A \xrightarrow{\rho_A} F \rightarrow 0,
\]
known as Atiyah sequence of $A$ (cf.~\cite{Kub}).
The main purpose of this paper is to investigate whether Atiyah classes of dg manifolds arising from regular Lie algebroids respect the above Atiyah sequence.
Our main theorem states that the Atiyah class of $A$ restricts to the Atiyah class of the bundle $K$ of Lie algebras along the inclusion $i$ on the one hand, and projects onto the Molino class~\cite{Molino} of the integrable distribution $F \subseteq T_M$ on the other hand (see Theorem~\ref{Thm:Funtoriality}).

For this purpose, observe that both $K$ and the quotient bundle $B := T_M/F$ carry canonical $A$-module structures. Thus, the pullback $\pi^\ast(E)$ of the Whitney sum $E:= K[1] \oplus B$ along the bundle projection $\pi \colon A[1] \to M$, equipped with the Chevalley-Eilenberg differential $d_{\CE}$ of the $A$-module $E$, is a dg vector bundle over $(A[1], d_A)$.
Combining results in~\cite{AC} and in~\cite{GsM},
there exists a contraction of dg vector bundles for the tangent dg vector bundle of $(A[1],d_A)$
\[
  \begin{tikzcd}
  (T_{A[1]}, L_{d_A}) \arrow[loop left, distance=2em, start anchor={[yshift=-1ex]west}, end anchor={[yshift=1ex]west}]{}{H} \arrow[r,yshift = 0.7ex, "\Phi"] &   (\E:= \pi^\ast(E), Q_\E:= d_{\CE} - \Omega) \arrow[l,yshift = -0.7ex, "\Psi"],
\end{tikzcd}
\]
where
\[
-\Omega \in \Omega_A^2(\Hom(B, K[1])) := \Gamma(\wedge^2A^\ast \otimes \Hom(B, K[1]))
\]
is a perturbation of the differential $d_{\CE}$.
Via this contraction, we give an explicit description of the Atiyah class of the dg manifold $(A[1],d_A)$ (see Theorem~\ref{Theorem on Atiyah class of A[1]}).
In particular, we obtain the Atiyah class of a bundle of Lie algebras (see Proposition~\ref{prop: Atiyah class of bundle of Lie algebras}), and rediscover the fact in~\cite{CXX} that the Atiyah class of the dg manifold arising from an integrable distribution $F \subseteq T_M$ is related via a canonical isomorphism to the Atiyah class of the Lie pair $(T_M, F)$ introduced in~\cite{CSX} (see Proposition~\ref{prop: Atiyah class of foliation}).
Since scalar Atiyah classes and Todd classes are generated by Atiyah classes, we also prove that all scalar Atiyah classes and Todd classes of a regular Lie algebroid $A$ respect the Atiyah sequence of $A$ (see Proposition~\ref{prop: functoriality on scalar Atiyah classes} and Proposition~\ref{prop: functoriality of Todd classes}).

Note that the perturbation term $-\Omega \in \Omega_A^2(\Hom(B,K[1]))$ is indeed $d_{\CE}$-closed. It was proved in~\cite{GsM} that the Chevalley-Eilenberg cohomology class
\[
\omega  = [\Omega] \in H_{\CE}^1(A;\Hom(B,K[1])) \cong  H_{\CE}^2(A;\Hom(B,K))
\]
is a characteristic class of $A$. In fact, the cohomology class $\omega$ measures whether the Lie algebroid structure locally split around leaves of the submanifold foliated from $F$.
More precisely, according to Theorem 7.3 in~\cite{GsM}, $\omega$ vanishes if and only if around each leaf $L$ foliated from $F$, there exists a tubular neighborhood $U \subset M$ of $L$ and an identification $U = L \times N$ such that the restriction $A\!\mid_{U}$ is isomorphic to the cross product of $A\!\mid_L$ and the trivial Lie algebroid over $N$.
For this reason, we say that a regular Lie algebroid \emph{locally split}, if the cohomology class $\omega$ vanishes.
For a locally splittable regular Lie algebroid $A$, we prove that the kernel $K$ of the anchor map is a Lie algebra bundle, and the Atiyah class of $A$ only consists of two components: one is
the Atiyah class of $K$ which is represented by the Lie bracket on $\Gamma(K)$; the other is the Atiyah class of its characteristic distribution $F$, which is related via a canonical isomorphism to the Atiyah class of the Lie pair $(T_M, F)$. The Todd class of $A$ is given by the product of the Todd class of this Lie algebra bundle $K$ that is represented by the Duflo element, and the Todd class of the characteristic distribution $F$ (see Theorem~\ref{Thm B}).

\paragraph{\bf Acknowledgement} 
I would like to thank Zhuo Chen, Yu Qiao, and Ping Xu for helpful discussions and comments.
I am also grateful to the two anonymous referees for valuable suggestions and comments to improve the presentation of the manuscript.

\section{Atiyah and Todd classes of dg vector bundles}
\subsection{Dg manifolds and dg vector bundles}
In this section, we briefly recall the definitions of dg manifolds and dg vector bundles. See~\cites{LMS, SSX} for details.

Let $M$ be a smooth manifold, and $\O_M$ be the sheaf of smooth functions over $M$.
By a ($\mathbb{Z}$-) graded manifold $\M$ with support $M$, we mean a pair $(M,\O_\M)$, where $\O_\M$ is a sheaf of $\mathbb{Z}$-graded commutative $\O_M$-algebras over $M$ such that for every contractible open subset $U \subset M$, $\O_\M(U)$ is locally isomorphic to $C^\infty(U) \otimes \widehat{SV}$ for a fixed finite dimensional $\mathbb{Z}$-graded vector space $V$. The space $C^\infty(\M) = \Gamma(\O_\M)$ of global sections of the structure sheaf $\O_\M$ will be referred as the algebra of smooth functions on $\M$.
A morphism $\phi$ of graded manifolds from $\M$ to $\N$ is a pair $(\phi_0,\phi^\ast)$, where $\phi_0\colon M \rightarrow N$ is a smooth map of base manifolds, and $\phi^\ast\colon \O_\N \rightarrow \O_\M$ is morphism of structure sheaves covering $\phi_0$.
In particular, $\phi^\ast$ induces a morphism of algebras of smooth functions $\phi^\ast\colon C^\infty(\N) \rightarrow C^\infty(\M)$.
The collection of graded manifolds and their morphisms constitute a category, called the category of graded manifolds.

A dg manifold is a graded manifold $\M$ together with a homological vector field, i.e., a vector field $Q_\M$ of degree $+1$ satisfying the integrable condition $[Q_\M, Q_\M] = 0$. In other words, the vector field $Q_\M$ arises from an infinitesimal action of the super Lie group $\R^{0\mid1}$ on the graded manifold $\M$. The associated cohomology
\[
 H^\bullet(\M) := H^\bullet(C^\infty(\M), Q_\M)
\]
is called the cohomology of the dg manifold $(\M,Q_\M)$.
A morphism $\phi$ of dg manifolds from $(\M, Q_\M)$ to $(\N, Q_\N)$ is a morphism of the underlying graded manifolds such that the morphism $\phi^\ast$ of algebras of smooth functions is a cochain map, i.e.,
\[
 \phi^\ast \circ Q_\N = Q_\M \circ \phi^\ast\colon C^\infty(\N) \rightarrow C^\infty(\M).
\]
The collection of dg manifolds together with their morphisms form a category, called the category of dg manifolds.

A dg vector bundle is an $\R^{0\mid1}$-equivariant vector bundle in the category of graded manifolds, or equivalently, a vector bundle in the category of dg manifolds, that is,  a graded vector bundle $\pi \colon \E \rightarrow \M$, equipped with two homological vector fields $Q_\E$ and $Q_\M$ coming from infinitesimal actions of $\R^{0\mid1}$ on $\E$ and $\M$, respectively, such that the bundle projection $\pi$ is morphism of dg manifolds. Here the $\R^{0\mid1}$-action on $\E$ is required to be fibre-preserving and fiberwise linear. In other words, the homological vector field $Q_\E$ is linear, with $Q_\M$ being its restriction onto the base manifold $\M$.
\begin{example}
Let $(\M, Q_\M)$ be a dg manifold. Then the vector bundle $(T^\ast_\M)^{\otimes p} \otimes (T_\M)^{\otimes q}$ together with the Lie derivative $L_{Q_\M}$ along the homological vector field $Q_\M$ is a dg vector bundle over $(\M, Q_\M)$ for all $p,q \geq 0$.
\end{example}

\subsection{Atiyah and Todd classes}
Let $(\E, Q_\E)$ be a dg vector bundle over a dg manifold $(\M, Q_\M)$. Mehta, Sti\'{e}non and Xu introduced in~\cite{MSX} (see also~\cite{LMS}) a cohomology class $\At_{(\E,Q_\E)}$, called the Atiyah class of $(\E,Q_\E)$, to measure the obstruction to the existence of a linear connection on the graded vector bundle $\E$ that is compatible with the homological vector fields $Q_\E$ and $Q_\M$.
More precisely, to any linear connection $\nabla$ on the graded vector bundle $\E$ is associated a degree $+1$ element $\At_{(\E,Q_\E)}^\nabla \in \Gamma(T^\ast_\M \otimes \End(\E))$ defined by
\begin{equation}\label{Atiyah cocycle}
 \At_{(\E,Q_\E)}^\nabla(X,e) = \Q_\E(\nabla_X e) - \nabla_{\Q_\M(X)} e - (-1)^{\abs{X}}\nabla_X Q_\E(e),
\end{equation}
for all homogeneous $X \in \Gamma(T_\M)$ and $e \in \Gamma(\E)$.
This element is closed under the homological vector field of the dg vector bundle $T^\ast_\M \otimes \End(\E)$. Its cohomology class
\[
\At_{(\E, Q_\E)} := [\At_{(\E,Q_\E)}^\nabla] \in H^1(\Gamma(T^\ast_\M \otimes \End(\E)))
\]
is independent of the choice of the linear connection $\nabla$, and is called the Atiyah class of the dg vector bundle $(\E, Q_\E)$.

For any positive integer $k$, one can form $\At^{k}_{(\E,Q_\E)}$, the image of $\At^{\otimes k}_{(\E,Q_\E)}$ under the natural map
\[
\otimes^k H^1(\Omega^1(\M) \otimes_{C^\infty(\M)} \Gamma(\End(\E))) \to H^k(\Omega^k(\M) \otimes_{C^\infty(\M)} \Gamma(\End(\E)))
\]
induced by the wedge product in the space $\Omega^\bullet(\M)$ of differential forms on $\M$ and the composition in $\End(\E)$.
The $k$-th scalar Atiyah class~\cite{MSX} of the dg vector bundle $(\E,Q_\E)$ is defined by
\[
  \ch_k(\E,Q_\E) := \frac{1}{k!}\left(\frac{i}{2\pi}\right)^k\str(\At_{(\E,Q_\E)}^k) \in H^k(\Omega^k(\M), L_{Q_\M}),
\]
where $\str \colon \Gamma(\End(\E)) \to C^\infty(\M)$ is the supertrace map (cf.~\cite{Manin}).

The Todd class of the dg vector bundle $(\E,Q_\E)$ is defined by
\[
 \Td_{(\E,Q_\E)} := \Ber\left(\frac{\At_{(\E,Q_\E)}}{1 - e^{-\At_{(\E,Q_\E)}}}\right) \in \prod_{k \geq 0} H^k(\Omega^k(\M), L_{Q_\M}),
\]
where $\Ber \colon \Gamma(\End(\E)) \to C^\infty(\M)$ is the Berezinian (or the superdeterminant) map (cf.~\cite{Manin}).

\subsection{Invariance under contractions}
A contraction of dg vector bundles over a dg manifold $(\M,Q_\M)$ from $(\E,Q_\E)$ to $(\F, Q_\F)$ is given by three bundle maps $(\phi,\psi, h)$ fitting in the following diagram
 \[
 \begin{tikzcd}
  (\E, Q_\E) \arrow[loop left, distance=2em, start anchor={[yshift=-1ex]west}, end anchor={[yshift=1ex]west}]{}{h} \arrow[r,yshift = 0.7ex, "\phi"] &   (\F, Q_\F) \arrow[l,yshift = -0.7ex, "\psi"],
\end{tikzcd}
\]
where both $\phi \colon (\E, Q_\E) \to (\F, Q_\F)$ and $\psi \colon (\F, Q_\F) \to (\E,Q_\E)$ are morphisms of dg vector bundles, and $h \colon \E \to \E[-1]$ is a bundle map, satisfying
\begin{align*}
  \phi \circ \psi &= \id_\F, & \psi \circ \phi &= \id_\E + [Q_\E, h] = \id_\E + Q_\E \circ h + h \circ Q_\E,
\end{align*}
and side conditions
\begin{align*}
  h^2 &= 0, & \phi \circ h &= 0, & h \circ \psi &= 0.
\end{align*}
Given a contraction $(\phi,\psi,h)$ of bounded dg vector bundles over $(\M,Q_\M)$ from $(\E,Q_\E)$ to $(\F, Q_\F)$, by taking dual we obtain a contraction $(\phi^\ast, \psi^\ast, h^\ast)$ of bounded dg vector bundles from $(\E^\ast, Q_{\E^\ast})$ to $(\F^\ast, Q_{\F^\ast})$.
Then applying the tensor product of these two contractions, we obtain the following contraction of bounded dg vector bundles over $(\M,Q_\M)$
  \[
 \begin{tikzcd}
  (\E^\ast \otimes \E \cong \End(\E), [Q_\E, -]) \arrow[loop left, distance=2em, start anchor={[yshift=-1ex]west}, end anchor={[yshift=1ex]west}]{}{H} \arrow[r,yshift = 0.7ex, "\Phi"] &   (\F^\ast \otimes F \cong \End(\F), [Q_\F,-]) \arrow[l,yshift = -0.7ex, "\Psi"],
\end{tikzcd}
\]
where the bracket $[-,-]$ means the graded commutator, and the three bundle maps are defined by
\begin{align*}
  \Phi &:= \psi^\ast \otimes \phi, & \Psi &:= \phi^\ast \otimes \psi, & H &:= (\psi\phi)^\ast\otimes h + h^\ast \otimes \id_\E.
\end{align*}
Passing to the cohomology of section spaces, we obtain an isomorphism
\begin{equation}\label{Eq: Isom from EndE to EndF}
  \Phi_{\E,\F}:= \id_{T^\ast_\M} \otimes \Phi \colon H^\bullet(\Gamma(T_\M^\ast \otimes \End(\E))) \xrightarrow{\quad \cong \quad} H^\bullet(\Gamma(T_\M^\ast \otimes \End(\F))).
\end{equation}
Atiyah classes of bounded graded vector bundles are invariant under contractions in the follow sense.
\begin{proposition}\label{prop: invariance of Atiyah class}
     Suppose that there exists a contraction $(\phi,\psi,h)$ of bounded dg vector bundles over $(\M, Q_\M)$ from $(\E, Q_\E)$ to $(\F, Q_\F)$. Then the isomorphism $\Phi_{\E,\F}$ in~\eqref{Eq: Isom from EndE to EndF} sends the Atiyah class $\At_{(\E,Q_\E)}$ of $(\E,Q_\E)$ to the Atiyah class  $\At_{(\F,Q_\F)}$ of $(\F,Q_\F)$.
\end{proposition}
\begin{proof}
Let us choose a $T_\M$-connection $\nabla^\F$ on the graded vector bundle $\F$. Via the inclusion $\psi \colon \F \to \E$ we may identify $\F$ as a subbundle of $\E$. For each direct sum decomposition $\E = \F \oplus \F^c$ of graded vector bundles, we choose a $T_\M$-connection $\nabla^c$ on the complement bundle $\F^c$ of $\F$ such that $\nabla^\F$ and $\nabla^c$ define a $T_\M$-connection $\nabla^\E$ on $\E$. This connection satisfies
\begin{equation}\label{Eq: connection on E and F}
  \nabla^\E_X \psi(v) = \psi (\nabla^\F_X v),
\end{equation}
for all $X \in \Gamma(T_\M)$ and all $v \in \Gamma(\F)$. Then for all homogeneous $X \in \Gamma(T_\M)$ and $v \in \Gamma(\F)$, we have
\begin{align*}
  &\quad \Phi_{\E,\F}(\At_{(\E,Q_\E)}^{\nabla^\E})(X, v) = \phi\left(\At_{(\E,Q_\E)}^{\nabla^\E}(X, \psi(v))\right) \\
    &= \phi\left(Q_\E (\nabla^\E_X \psi(v)) - \nabla^\E_{[Q_\M,X]}\psi(v) - (-1)^{\abs{X}}\nabla^\E Q_\E(\psi(v)) \right) \quad (\text{since $\psi$ is a cochain map})\\
  &= \phi\left(Q_\E (\nabla^\E_X \psi(v)) - \nabla^\E_{[Q_\M,X]}\psi(v) - (-1)^{\abs{X}}\nabla^\E \psi(Q_\F(v)) \right) \quad (\text{by Eq.~\eqref{Eq: connection on E and F}})\\
  &= \phi\left(Q_\E(\psi(\nabla^\F_X v)) - \psi(\nabla^\F_{[Q_\M,X]}v) - (-1)^{\abs{X}}\psi(\nabla^\F_X Q_\F(v))  \right) \quad  (\text{since $\psi$ is a cochain map})\\
  &= (\phi\circ\psi)\left( Q_\F(\nabla^\F_X v) - \nabla^\F_{[Q_\M,X]}v - (-1)^{\abs{X}} \nabla^\F_X Q_\F(v)\right) \\
  &=  Q_\F(\nabla^\F_X v) - \nabla^\F_{[Q_\M,X]}v - (-1)^{\abs{X}} \nabla^\F_X Q_\F(v) = \At_{(\F,Q_\F)}^{\nabla^\F}(X,v).
\end{align*}
Passing to the cohomology level, we conclude the proof.
\end{proof}

As an immediate consequence, we have the following invariance on scalar Atiyah and Todd classes.
\begin{corollary}\label{corollary on invariance of scalar Atiyah and Todd}
  Assume that there exists a contraction of bounded dg vector bundle from $(\E, Q_\E)$ to $(\F, Q_\F)$. Then we have for all positive integer $k$,
  \begin{align*}
    \ch_k(\E,Q_\E) &= \ch_k(\F, Q_\F),
 \end{align*}
 and
 \begin{align*}
     \Td_{(\E, Q_\E)} &= \Td_{(\F,Q_\F)}.
  \end{align*}
\end{corollary}

\section{Cohomology of regular Lie algebroids}
In this section, we study the cohomology of tensor fields on the dg manifold $(A[1], d_A)$ arising from a regular Lie algebroid $A$.
\subsection{The graded geometry of regular Lie algebroids}
\subsubsection{The graded geometry of Lie algebroids}
We start with a brief discussion on the graded geometry of general Lie algebroids.
According to Va\u{i}ntrob~\cite{Vai}, a Lie algebroid structure on a smooth vector bundle $A$ is one-to-one correspondent to a homological vector field $d_A$ on the graded manifold $A[1]$.
More precisely, given a Lie algebroid $A$ over $M$, we obtain a dg manifold $(A[1], d_A)$ whose space of functions is the Chevalley-Eilenberg dg algebra $(\Omega_A = \Gamma(\wedge A^\ast), d_A)$ of $A$.
Furthermore, each $A$-module $(E, \nabla^E)$ consisting of a ungraded vector bundle $E \to M$ and a flat Lie algebroid $A$-connection $\nabla^E$ gives rise to a dg vector bundle over the dg manifold $(A[1],d_A)$ via the following pullback diagram
\[
 \begin{tikzcd}
   (\E:= \pi^\ast E, d^{\nabla^E}) \ar[r] \ar[d] & E \ar[d] \\
    (A[1], d_A) \ar{r}{\pi} &   M.
 \end{tikzcd}
\]
However, the converse is not true in general, since any graded vector bundle over the graded manifold $A[1]$ does not necessarily arise from the pullback of an ungraded vector bundle over $M$ along $\pi$, but rather from the pullback of a graded vector bundle over $M$ (cf.~\cite{Mehta}).
Indeed, dg vector bundles over the dg manifold $(A[1],d_A)$ are one-to-one correspondent to representations up to homotopy (or $\infty$-representations) of $A$ consisting of graded vector bundles over $M$ equipped with flat $A$-superconnections~\cites{AC,GsM}.
\begin{definition}[\cites{AC,GsM}]
  A representation up to homotopy, or an $\infty$-representation, of $A$ on a graded vector bundle $E = \oplus_{i\in \mathbb{Z}} E^i$ over $M$ is a square zero operator $D$ of degree $1$ on
  \[
   \Omega_A(E) := \Omega_A \otimes_{C^\infty(M)} \Gamma(E)  = \Gamma(\wedge A^\ast \otimes E),
  \]
  satisfying the Leibniz rule
  \[
    D(\alpha\omega) = (d_A\alpha)\omega + (-1)^{\abs{\alpha}} \alpha(D\omega),
  \]
  for all homogeneous $\alpha \in \Omega_A$ and $\omega \in \Omega_A(E)$.
  The cohomology of the resulting complex is denoted by
  \[
  H^\bullet(A; E) = \bigoplus_k H^k(A; E) := \bigoplus_{p+q =k} H^k(\Omega_A^p(E^q), D).
  \]
\end{definition}
It is proved in~\cites{AC, GsM} that there exists a correspondence between the tangent dg vector bundle $(T_{A[1]}, L_{d_A})$ of $(A[1], d_A)$ and a representation up to homotopy of the Lie algebroid $A$ on the graded vector bundle $A[1] \oplus T_M$ over $ M$, called the adjoint representation in~\cite{AC}, which we now recall.

Observe that there is a short exact sequence of vector bundles over the graded manifold $A[1]$:
\begin{equation}\label{SES of tangent bundle on F[1]}
	0 \to \pi^\ast(A[1]) \xrightarrow{I}  T_{A[1]} \xrightarrow{\pi_\ast}  \pi^\ast(T_M) \to 0,
\end{equation}
where $\pi_\ast\colon T_{A[1]} \to \pi^\ast(T_M)$ is the tangent map of the bundle projection $\pi \colon A[1] \to M$, and $I$ is the canonical vertical lifting.
Taking global sections gives rise to a short exact sequence of left graded $\Omega_A$-modules:
\[
	0 \to  \Omega_A \otimes_{C^\infty(M)} \Gamma(A[1]) \xrightarrow{I}  \Gamma(T_{A[1]}) \xrightarrow{\pi_\ast}  \Omega_A \otimes_{C^\infty(M)} \Gamma(T_M) \to 0.
\]
The canonical vertical lifting $I$ is given by the $\Omega_A$-linear contraction defined by for all $\omega \in \Omega_A$ and $a[1] \in \Gamma(A[1])$,
\begin{equation}\label{Eq: Def of I}
I(\omega \otimes a[1]) = \omega \otimes \iota_a \in \Der(\Omega_A) \cong \Gamma(T_{A[1]}).
\end{equation}
Let us choose a linear connection $\nabla^A$ on the vector bundle $A$ over $M$.
This connection $\nabla^A$ induces a splitting of the short exact sequence~\eqref{SES of tangent bundle on F[1]} such that $T_{A[1]} \cong A[1] \times_{M} (A[1] \oplus T_M)$. Thus, one has an isomorphism of $\Omega_A$-modules
\begin{align}\label{Eqt:splitXFone}
 \Gamma(T_{A[1]}) &\stackrel{\cong}{\longrightarrow} \Omega_A (A[1] \oplus T_M) = \Omega_A \otimes_{C^\infty(M)} \Gamma(A[1] \oplus T_M).
\end{align}
The isomorphism~\eqref{Eqt:splitXFone} transfers the Lie derivative $L_{d_A}$ on $\Gamma(T_{A[1]})$ to a square zero derivation
\begin{equation}\label{Eq: Def of DnablaF}
 D_{\nabla^A} = \left(
                    \begin{array}{cc}
                      d_{\nabla^{\bas}} & R^{\bas}_\nabla \\
                      \rho_A & d_{\nabla^{\bas}} \\
                    \end{array}
                  \right) \colon
                  \left(
                    \begin{array}{c}
                      \Omega^\bullet_A(A[1]) \\
                      \Omega^{\bullet-1}_A(T_M) \\
                    \end{array}
                  \right) \rightarrow
                  \left(
                    \begin{array}{c}
                      \Omega^{\bullet+1}_A(A[1]) \\
                      \Omega^{\bullet}_A(T_M) \\
                    \end{array}
                    \right),
\end{equation}
where
\begin{itemize}
	\item $\rho_A$ is the $\Omega_A$-linear extension of the anchor map $\rho_A \colon A[1] \to T_M$;
	\item $d_{\nabla^{\bas}} \colon \Gamma(A[1]\oplus T_M) \to \Omega^1_A(A[1] \oplus T_M)$ is the covariant derivative of the basic $A$-connection $\nabla^{\bas}$ on $A[1] \oplus T_M$ defined by
   \[
     \nabla^{\bas}_a(u) := \rho_A(\nabla^A_u a) + [\rho_A(a), u],
   \]
   and
   \[
    \nabla^{\bas}_a(a^\prime[1]):= \left(\nabla^A_{\rho_A(a^\prime)} a + [a,a^\prime]_A\right) [1],
   \]
  for all $a, a^\prime \in \Gamma(A)$ and $u \in \Gamma(T_M)$;
\item ${R^{\bas}_{\nabla^A}} \in \Omega_A^2(\operatorname{Hom}(T_M, A[1]))$, called the basic curvature of $\nabla^A$, defines an $\Omega_A$-linear map $\Omega_A^{\bullet-1}(T_M) \to \Omega_{A}^{\bullet+1}(A[1])$ by
	 	\begin{align*}
 {R^{\bas}_{\nabla^A}}(u)(a^\prime, a^{\prime\prime}) &   :=\left(\nabla^A_{u}[a^\prime, a^{\prime\prime}]_A - [\nabla^A_u a^\prime, a^{\prime\prime}]_A - [a^\prime, \nabla^A_u  a^{\prime\prime}]_A - \nabla^A_{\nabla^{\bas}_{a^{\prime\prime}}u}a^\prime +\nabla^A_{\nabla^{\bas}_{a^\prime} u} a^{\prime\prime}\right)[1],
	 	\end{align*}
for all $a^\prime, a^{\prime\prime} \in \Gamma(A)$, and $u \in \Gamma(T_M)$.
	 \end{itemize}
In particular, we obtain an isomorphism of cochain complexes
\begin{equation}\label{Eq: TA as representation UTH}
  	\big(\Gamma(T_{A[1]}), L_{d_A}\big) \xrightarrow{\;\cong\;} \big(\Omega_A(A[1] \oplus T_M), D_{\nabla^A} = \rho_A + d_{\nabla^{\bas}} + R_{\nabla^A}^{\bas}\big).
\end{equation}
By taking dual operation, we obtain an isomorphism
\[
 \big(\Omega^1(A[1]) = \Gamma(T^\ast_{A[1]}), L_{d_A}\big) \xrightarrow{\;\cong\;} \big(\Omega_A(T^\ast M \oplus A^\ast[-1]),  D^\ast_{\nabla^A} = \rho_A^\ast + d^\ast_{\nabla^{\bas}} - (R^{\bas}_{\nabla^A})^\ast\big),
\]
where $\rho_A^\ast \colon \Omega_A(T^\ast M) \to \Omega_A(A^\ast[-1])$ and $(R^{\bas}_{\nabla^A})^\ast \in \Omega_A^2(\Hom(A^\ast[-1], T^\ast M))$ are the $\Omega_A$-linear dual of $\rho_A$ and $R^{\bas}_{\nabla^A}$, respectively, while $d^\ast_{\nabla^{\bas}}$ is the covariant derivative of the dual $A$-connection on $T^\ast M \oplus A^\ast[-1]$.

\subsubsection{Applications to regular Lie algebroids}
Let $(A, \rho_A, [-,-]_A)$ be a regular Lie algebroid over $M$ with the characteristic distribution $F:= \operatorname{Im}\rho_A \subseteq {T_M}$ of constant rank.
Let $K \subseteq A$ be the kernel of $\rho_A$, which is a bundle of Lie algebras over $M$ such that the inclusion $(\Gamma(K), [-,-]_K) \hookrightarrow (\Gamma(A), [-,-]_A)$ is a morphism of Lie algebras. We thus obtain a short exact sequence of Lie algebroids over $M$
\begin{equation}\label{Eq: SES of A}
  0 \to K \xrightarrow{i} A \xrightarrow{\rho_A} F \to 0,
\end{equation}
known as the Atiyah sequence of the regular Lie algebroid $A$ (\cite{Kub}).
According to~\cite{GsM}, the differential $D_{\nabla^A}$~\eqref{Eq: Def of DnablaF} can be simplified if we choose a special linear connection on the regular Lie algebroid $A$.
Note that there is another short exact sequence of vector bundles over $M$
\begin{equation}\label{SES}
0 \rightarrow F \rightarrow T_M \xrightarrow{\pr_B} B \rightarrow 0,
\end{equation}
where we denote by $B$ the quotient bundle $T_M/F$, which can be thought of as the normal bundle of the characteristic foliation $\mathcal{F} \subseteq M$.
We fix a quadruple
\begin{equation}\label{Eq: quadruple}
(\tau, j, \nabla^K, \nabla^F),
\end{equation}
where
\begin{itemize}
  \item $\tau \colon F \to A$ is a splitting of the short exact sequence~\eqref{Eq: SES of A} of vector bundles, and $j \colon B \to T_M$ is a splitting of the short exact sequence~\eqref{SES} of vector bundles;
  \item $\nabla^K$ is a linear connection on $K$ extending the $F$-connection on $K$ defined by
   \[
        \nabla^K_{u_F} \xi = [\tau(u_F), a_K]_A,
   \]
   for all $u_F \in \Gamma(F)$ and $a_K \in \Gamma(K)$;
   \item $\nabla^F$ a linear connection on $F$ extending a torsion-free $F$-connection on $F$, satisfying
   \[
    \nabla^F_{b} u_F = \pr_F[j(b), u_F],
   \]
   for all $b \in \Gamma(B), u_F \in \Gamma(F)$.
\end{itemize}
Let
 \begin{equation}\label{Eq: nablaA}
 \nabla^A_u a = \nabla^K_{u}(a_K) + \tau(\nabla^F_{u}a_F) + R^\tau(u_F, a_F),
\end{equation}
 for all $u \in \Gamma(T_M)$ and $a \in \Gamma(A)$, where $u_F$ is the component of $u$ in $\Gamma(F)$, $a_K$ and $a_F$ are the components of $a$ in $\Gamma(K)$ and $\Gamma(F)$, respectively, and $R^\tau \in \Omega_F^2(K)$, called the curvature of the splitting $\tau$, is defined by for all $u_F,v_F \in \Gamma(F)$,
\begin{align*}
   R^\tau(u_F, v_F) &= [\tau(u_F),\tau(v_F)]_A - \tau([u_F,v_F]).
\end{align*}
It is easy to see that $\nabla^A$ defined above is indeed a linear connection on $A$.
\begin{lemma}[\cite{GsM}]\label{lem: special connection on F}
The basic connection $\nabla^{\bas}$ and the basic curvature $R^{\bas}_{\nabla^A} \in \Omega^2_A (\Hom(T_M, A))$ of the linear connection $\nabla^A$ defined in~\eqref{Eq: nablaA} satisfies
\begin{align*}
  \nabla^{\bas}_a u &= \nabla^F_{\rho_A(a)} u_F + \nabla^{\Bott}_{\rho_A(a)} u_B = \nabla^F_{\rho_A(a)} u_F + \pr_B[\rho_A(a), j(u_B)], \\
  \nabla^{\bas}_a a^\prime &= [a, a_K^\prime]_A + \tau(\nabla^F_{\rho_A(a)}\rho_A(a^\prime)), \\
 {R^{\bas}_{\nabla^A}}(a,a^\prime)(u_F) &= -\tau (R^{\nabla^F}(\rho_A(a), \rho_A(a^\prime))u_F),
 \end{align*}
 and
\begin{align*}
R^{\bas}_{\nabla^A}(a,a^\prime)&(j(b)) = \nabla^K_{j(b)}[a_K, a_K^\prime]_A - [\nabla^K_{j(b)} a_K, a_K^\prime]_A - [a_K, \nabla^K_{j(b)} a_K^\prime]_A \\
&- R^{\nabla^K}(\rho_A(a), j(b))a_K^\prime + R^{\nabla^K}(\rho_A(a^\prime), j(b))a_K  \\
&+ \nabla^K_{j(b)} R^\tau(\rho_A(a), \rho_A(a^\prime)) - R^\tau(\pr_F[j(b), \rho_A(a)], \rho_A(a^\prime)) - R^\tau(\rho_A(a), \pr_F[j(b), \rho_A(a^\prime)]),
\end{align*}
for all $a, a^\prime, a^{\prime\prime} \in \Gamma(A), u \in \Gamma(T_M)$ and $b \in \Gamma(B)$, where $u_F$ and $u_B$ are the components of $u$ in $\Gamma(F)$ and $\Gamma(B)$, respectively, $a_K = \pr_K(a)$ is the component of $a$ in $\Gamma(K)$, $R^{\nabla^F}$ and $R^{\nabla^K}$ are the curvatures of the chosen linear connections $\nabla^F$ and $\nabla^K$, respectively.
 \end{lemma}
As a consequence, the basic curvature defines an $\Omega_A$-linear map of degree $1$
\[
\Omega \colon  \Omega^\bullet_A(B) \to \Omega_A^{\bullet+2}(K[1])
\]
by
\begin{equation}\label{Eq: Def of Omega}
\Omega(b)(a,a^\prime) := -R^\bas_{\nabla^A}(a,a^\prime)(j(b))[1],
\end{equation}
for all $a, a^\prime \in \Gamma(A)$ and $b \in \Gamma(B)$.

\subsection{Contractions of tangent dg vector bundles}
Given a regular Lie algebroid $(A,\rho_A,[-,-]_A)$, there exist ordinary representations of $A$ on the vector bundles $K$ and $B$, that is, flat Lie algebroid $A$-connections on $K$ and $B$ defined by
\begin{align}\label{Eq: A-module on K}
  \widetilde{\nabla}^K &\colon \Gamma(A) \times \Gamma(K) \to \Gamma(K),  &
  \widetilde{\nabla}^K_a a_K &= [a, a_K]_A,
\end{align}
and
\begin{align}\label{Eq: A-module on B}
    \widetilde{\nabla}^B &\colon \Gamma(A) \times \Gamma(B) \to \Gamma(B), &
  \widetilde{\nabla}^B_a b &= \pr_B([\rho_A(a), j(b)]),
\end{align}
respectively, for all $a \in \Gamma(A), a_K \in \Gamma(K), b \in \Gamma(B)$. Here $j \colon B \to T_M$ is any splitting of the short exact sequence~\eqref{SES} of vector bundles.

According to Example 4.17 in Abad and Crainic~\cite{AC}, the adjoint representation of the regular Lie algebroid $A$ on the 2-term complex $A[1] \xrightarrow{\rho_A} T_M$ is quasi-isomorphic to a representation up to homotopy on its cohomology $K[1] \to B$. This representation up to homotopy on $K[1] \oplus B$ is indeed a perturbation of the ordinary representation of $A$ defined in~\eqref{Eq: A-module on K} and~\eqref{Eq: A-module on B}.
However, the explicit construction was skipped.  For completeness, we give a thorough description in the following proposition.
\begin{proposition}[\cite{AC}]\label{prop: contraction on TA}
Let $A$ be a regular Lie algebroid over $M$. For any quadruple $(\tau, j, \nabla^K, \nabla^F)$ as in~\eqref{Eq: quadruple}, there is a contraction for the adjoint representation of $A$ on the graded vector bundle $A[1] \oplus T_M$
\begin{equation}\label{Eq: contraction for RUTH}
\begin{tikzcd}
	(\Omega_A(A[1] \oplus T_M), D_{\nabla^A}) \arrow[loop left, distance=2em, start anchor={[yshift=-1ex]west}, end anchor={[yshift=1ex]west}]{}{h} \arrow[r,yshift = 0.7ex, "\phi"] & (\Omega_A(K[1] \oplus B), d_{\CE} - \Omega) \arrow[l,yshift = -0.7ex, "\psi"],
\end{tikzcd}
\end{equation}
where $\nabla^A$ is the linear connection on $A$ defined in~\eqref{Eq: nablaA},  $D_{\nabla^A}$ is the differential defined in~\eqref{Eq: Def of DnablaF}, $d_{\CE}$ is the Chevalley-Eilenberg differential of the $A$-module $K[1] \oplus B$ defined in~\eqref{Eq: A-module on K} and~\eqref{Eq: A-module on B}, and $\Omega \colon \Omega^\bullet_A(B) \to \Omega_A^{\bullet+2}(K[1])$ is defined in~\eqref{Eq: Def of Omega}.
\end{proposition}
To prove this proposition, we need the following perturbation lemma (see~\cite{Manetti}).
\begin{lemma}[Perturbation lemma]\label{Lem: OPT}
Given a contraction of cochain complexes
\[
\begin{tikzcd}
(P, \delta) \arrow[loop left, distance=2em, start anchor={[yshift=-1ex]west}, end anchor={[yshift=1ex]west}]{}{h } \arrow[r,yshift = 0.7ex, "\phi"] & (T , d) \arrow[l,yshift = -0.7ex, "\psi"],
\end{tikzcd}
\]
and a perturbation $\varrho$ of the differential $\delta$, i.e.,  a linear map $\varrho\colon P \to P[1]$ satisfying $\delta+\varrho$ is a new differential on $P$ and the following constraints
 \begin{align}\label{Eq: perturb constraints}
    \cup_n \ker((h\varrho)^n\psi) &= T, & \cup_n \ker(\phi(\varrho h)^n) &= P, & \cup_n \ker(h(\varrho h)^n) &= P,
 \end{align}
 the series
\begin{align}
	\vartheta &:=\sum_{k=0}^\infty \phi(h\varrho)^k\varrho\psi, & 	\phi_\flat &:=\sum_{k=0}^\infty \phi(\varrho h)^k, \label{Eq: Def of thetaandphi}\\
	\psi_\flat &:=\sum_{k=0}^\infty (h\varrho)^k \psi, & h_\flat &:=\sum_{k=0}^\infty h (\varrho h)^k \label{Eq: Def of psiandh}
\end{align}
all converge, and the datum
\[
	\begin{tikzcd}
	(P, \delta  +\varrho) \arrow[loop left, distance=2em, start anchor={[yshift=-1ex]west}, end anchor={[yshift=1ex]west}]{}{h_\flat } \arrow[r,yshift = 0.7ex, "\phi_\flat"] & (T ,  d+\vartheta) \arrow[l,yshift = -0.7ex, "\psi_\flat"]
	\end{tikzcd}
\]
constitutes a new contraction.
\end{lemma}

\begin{proof}[Proof of Proposition~\ref{prop: contraction on TA}]
Given a splitting $j \colon B \to T_M$ and $\tau \colon F \to A$ of the short exact sequences~\eqref{Eq: SES of A} and~\eqref{SES} of vector bundles, respectively, there is a contraction of cochain complexes of $C^\infty(M)$-modules
  \[
   	\begin{tikzcd}
	(\Gamma(A[1] \oplus T_M), \rho_A) \arrow[loop left, distance=2em, start anchor={[yshift=-1ex]west}, end anchor={[yshift=1ex]west}]{}{h} \arrow[r,yshift = 0.7ex, "\phi"] & (\Gamma(K[1] \oplus B), 0), \arrow[l,yshift = -0.7ex, "\psi"]
	\end{tikzcd}
  \]
where the three maps $(\phi, \psi, h)$ are defined  by
\begin{align*}
    \phi(a[1] + u) &= \pr_K(a)[1] + \pr_B(u), & \psi(a_K[1] + b) &= a_K[1] + j(b), & h(a[1] + u) &= -\tau(\pr_F(u))[1],
  \end{align*}
for all $a \in \Gamma(A), a_K \in \Gamma(K)$ and $u \in \Gamma(T_M)$.
The $\Omega_A$-linear extension of this contraction gives rise to a contraction of $\Omega_A$-modules
\begin{equation}\label{Eq: contraction of delta}
\begin{tikzcd}
	(\Omega_A(A[1] \oplus T_M), \rho_A) \arrow[loop left, distance=2em, start anchor={[yshift=-1ex]west}, end anchor={[yshift=1ex]west}]{}{h} \arrow[r,yshift = 0.7ex, "\phi"] & (\Omega_A(K[1] \oplus B) , 0) \arrow[l,yshift = -0.7ex, "\psi"].
\end{tikzcd}
\end{equation}
Observe that $d_{\nabla^{\bas}} + R^{\bas}_{\nabla^A}$ defines a perturbation of the differential $\rho_A$.
Since
\begin{align*}
 h((d_{\nabla^{\bas}}+ R^{\bas}_{\nabla^A})(a[1]+u)) &= h(d_{\nabla^{\bas}}(u)) \in \Omega_A^1(A[1])
\end{align*}
is defined by
\begin{align*}
  \iota_ah(d_{\nabla^{\bas}}(u)) &= h(\nabla^{\bas}_a u) = h(\nabla^F_{\rho_A(a)}u_F + \pr_B[\rho_A(a), j(u_B)]) \\
  &= h(\nabla^F_{\rho_A(a)}u_F) = -\tau(\nabla^F_{\rho_A(a)}u_F)[1],
\end{align*}
and
\begin{align*}
  (d_{\nabla^{\bas}} + R^{\bas}_{\nabla^A})(h(a[1]+u)) &= -d_{\nabla^{\bas}}(\tau(u_F)) \in \Omega_A^1(A[1])
\end{align*}
is defined by
\begin{align*}
  \iota_a(-d_{\nabla^{\bas}}(\tau(u_F))) &= -\nabla^{\bas}_a(\tau(u_F)) = -\tau(\nabla^F_{\rho_A(a)} u_F)[1],
\end{align*}
it follows that
\[
   h \circ (d_{\nabla^{\bas}} + R^{\bas}_{\nabla^A}) = - (d_{\nabla^{\bas}} + R^{\bas}_{\nabla^A}) \circ h \colon \Omega_A(A[1] \oplus T_M) \to \Omega_A(A[1] \oplus T_M).
\]
Combining with the side conditions $\phi \circ h = 0, h \circ \psi = 0$ and $h^2 = 0$, we have
\begin{align*}
 \phi \circ  (d_{\nabla^{\bas}} + R_{\nabla^A}^{\bas}) \circ h &= - \phi \circ h \circ (d_{\nabla^{\bas}} + R_{\nabla^A}^{\bas}) = 0, \\
 h \circ (d_{\nabla^{\bas}} + R_{\nabla^A}^{\bas}) \circ \psi &= - (d_{\nabla^{\bas}} + R_{\nabla^A}^{\bas}) \circ h \circ \psi = 0, \\
 h \circ (d_{\nabla^{\bas}} + R_{\nabla^A}^{\bas}) \circ h &= h^2 \circ  (d_{\nabla^{\bas}} +  R_{\nabla^A}^{\bas}) = 0.
\end{align*}
Thus, the maps $\phi, \psi, h$, and the perturbation $d_{\nabla^{\bas}} + R_{\nabla^A}^{\bas}$ satisfy constraints in Equation~\eqref{Eq: perturb constraints}.
Applying the perturbation Lemma~\ref{Lem: OPT} to Contraction~\eqref{Eq: contraction of delta} and the perturbation $d_{\nabla^{\bas}} + R_{\nabla^A}^{\bas}$, we obtain a new contraction
\[
\begin{tikzcd}
	(\Omega_A(A[1] \oplus T_M), D_{\nabla^A} = \rho_A + d_{\nabla^{\bas}} + R_{\nabla^A}^{\bas}) \arrow[loop left, distance=2em, start anchor={[yshift=-1ex]west}, end anchor={[yshift=1ex]west}]{}{h^\prime} \arrow[r,yshift = 0.7ex, "\phi^\prime"] & (\Omega_A(K[1] \oplus B) , D^\prime) \arrow[l,yshift = -0.7ex, "\psi^\prime"],
\end{tikzcd}
\]
where
\begin{align*}
  \phi^\prime &= \sum_{l =0}^{\infty}\phi ((d_{\nabla^{\bas}} + R^{\bas}_{\nabla^A}) h)^l = \phi + \sum_{l=1}^\infty \big(\phi (d_{\nabla^{\bas}}  + R^{\bas}_{\nabla^A}) h\big) \circ \big((d_{\nabla^{\bas}} + R^{\bas}_{\nabla^A}) h\big)^{l-1} = \phi, \\
  \psi^\prime &= \sum_{l=0}^{\infty}(h (d_{\nabla^{\bas}} + R^{\bas}_{\nabla^A}))^l \psi = \psi + \sum_{l=1}^{\infty} (h (d_{\nabla^{\bas}} + R^{\bas}_{\nabla^A}))^{l-1}  \circ \big(h  (d_{\nabla^{\bas}} + R^{\bas}_{\nabla^A})  \psi\big) = \psi, \\
  h^\prime &= \sum_{l=0}^{\infty} (h (d_{\nabla^{\bas}} + R^{\bas}_{\nabla^A}))^l  h = h + \sum_{l=1}^{\infty} (h (d_{\nabla^{\bas}} + R^{\bas}_{\nabla^A}))^{l-1} \circ \big(h  (d_{\nabla^{\bas}} + R^{\bas}_{\nabla^A})  h\big) = h,
  \end{align*}
  and the new differential $D^\prime$ on $\Omega_A(K[1] \oplus B)$ is
  \begin{align*}
  D^\prime &= \phi^\prime \circ (d_{\nabla^{\bas}} + R^{\bas}_{\nabla^A}) \circ \psi = d_{\CE} - \Omega,
\end{align*}
according to Lemma~\ref{lem: special connection on F}.
\end{proof}

Combining Equation~\eqref{Eq: TA as representation UTH} with Proposition~\ref{prop: contraction on TA}, we obtain a contraction for the tangent dg vector bundle of the dg manifold $(A[1], d_A)$.
\begin{corollary}\label{corollary on TA}
  For each choice of quadruple $(j,\tau,\nabla^K, \nabla^F)$, there exists a contraction of dg vector bundles over the dg manifold $(A[1], d_A)$
  \[
    \begin{tikzcd}
	(T_{A[1]}, L_{d_A}) \arrow[loop left, distance=2em, start anchor={[yshift=-1ex]west}, end anchor={[yshift=1ex]west}]{}{H} \arrow[r,yshift = 0.7ex, "\Phi"] & (\E, Q_\E), \arrow[l,yshift = -0.7ex, "\Psi"]
	\end{tikzcd}
  \]
  where $\E = \pi^\ast(E)$ is the pullback bundle of the Whitney sum $E:= K[1] \oplus B$ over $M$ along the projection $\pi \colon A[1] \to M$, $Q_\E = d_{\CE} - \Omega$ is the differential on $\Gamma(\E)$ induced by the pair of homological vector fields carried by the vector bundle $\E$ and its base manifold $A[1]$,
  and the inclusion $\Psi$ is defined by
  \begin{align}\label{Eq: Def of Psi}
    \Psi(\alpha \otimes a_K[1]) &= I(\psi(\alpha \otimes a_K[1])) = \alpha \otimes \iota_{a_K}, & \Psi(\alpha \otimes b) &= \alpha \otimes \nabla^A_{j(b)},
  \end{align}
  for all $\alpha \in \Omega_A, a_K \in \Gamma(K)$ and $b \in \Gamma(B)$.
  Here $\nabla^A$ is the linear connection on $A$ defined in~\eqref{Eq: nablaA} and $I \colon \Omega_A(A[1]) \to \Gamma(T_{A[1]})$ is the canonical vertical lifting in~\eqref{Eq: Def of I}.
 \end{corollary}
\begin{remark}
 When the anchor $\rho_A$ is injective, i.e., $A$ is identified with its characteristic distribution $F = \operatorname{Im}\rho_A$, the above contraction was explicitly constructed in~\cite{CXX} to compute Atiyah and Todd classes of integrable distributions.
\end{remark}

As a consequence, the map $\Psi$ induces an isomorphism from the cohomology of the representation up to homotopy $(\Omega_A(K[1]\oplus B), Q_\E)$ of $A$ to the cohomology of the tangent dg vector bundle $(T_{A[1]}, L_{d_A})$, i.e.,
 \[
   \Psi \colon H^\bullet(A; K[1] \oplus B) := H^\bullet(\Omega_A(K[1] \oplus B), Q_\E) \xrightarrow{\cong} H^\bullet(\Gamma(T_{A[1]}), L_{d_A}).
 \]
By taking dual operation, we obtain a contraction for the cotangent dg vector bundle of the dg manifold $(A[1],d_A)$.
\begin{corollary}\label{corollary on TastA}
  For each choice of quadruple $(j,\tau,\nabla^K, \nabla^F)$, there exists a contraction for the cotangent dg vector bundle of the dg manifold $(A[1],d_A)$
  \[
  \begin{tikzcd}
	(T^\ast_{A[1]}, L_{d_A}) \arrow[loop left, distance=2em, start anchor={[yshift=-1ex]west}, end anchor={[yshift=1ex]west}]{}{H^\ast} \arrow[r,yshift = 0.7ex, "\Psi^\ast"] & (\E^\ast  = \pi^\ast(B^\ast \oplus K^\ast[-1]), Q_{\E^\ast}:= d_{\CE} + \Omega^\ast),  \arrow[l,yshift = -0.7ex, "\Phi^\ast"]
	\end{tikzcd}
  \]
  where $d_{\CE}$ is the Chevalley-Eilenberg differential of the dual $A$-module $B^\ast \oplus K^\ast[-1]$, and $\Omega^\ast$ is the $\Omega_A$-linear dual of $\Omega$.
\end{corollary}
Note that tensor products of the two contractions in Corollary~\ref{corollary on TA} and Corollary~\ref{corollary on TastA} produce contractions from the space $\Gamma(T^{\ast \otimes m}_{A[1]} \otimes T_{A[1]}^{\otimes n})$ of $(m,n)$ tensor fields on the dg manifold $(A[1], d_A)$ to the section space $\Gamma(\E^{\ast \otimes m} \otimes \E^{\otimes n})$ of the dg vector bundle $(\E, Q_\E)$.
For simplicity,  we will also denote by $Q_\E$ the induced differential on the section space $\Gamma(\E^{\ast \otimes m} \otimes \E^{\otimes n})$ by abuse of notation.
When passing to cohomology, for all pairs $(m,n)$ of non-negative integers, we obtain isomorphisms
 \begin{equation}\label{Eq: isom for mn tensor fields}
    \Phi_{m,n} \colon H^\bullet(\Gamma(T^{\ast \otimes m}_{A[1]} \otimes T_{A[1]}^{\otimes n}), L_{d_A}) \xrightarrow{\quad \cong \quad} H^\bullet(\Gamma(\E^{\ast \otimes m} \otimes \E^{\otimes n}), Q_\E) = H^\bullet(A; E^{\ast \otimes m} \otimes E^{\otimes n}).
 \end{equation}
Here $H^\bullet(A; E^{\ast \otimes m} \otimes E^{\otimes n})$ is the cohomology of the representation up to homotopy of $A$ on the tensor products $E^{\ast \otimes m} \otimes E^{\otimes n}$ of graded vector bundles over $M$.

\section{Atiyah classes of regular Lie algebroids}
In this section, we study the Atiyah class of a regular Lie algebroid $A$ over $M$, that is, the Atiyah class of the tangent dg vector bundle $(T_{A[1]}, L_{d_A})$ of the dg manifold $(A[1], d_A)$ arising from $A$.
\subsection{Atiyah classes}
Let us fix a quadruple $(\tau, j,\nabla^K,\nabla^F)$ as in~\eqref{Eq: quadruple} and denote by $\nabla^A$ the associated linear connection on $A$ defined in~\eqref{Eq: nablaA}.
By Proposition~\ref{prop: invariance of Atiyah class} and Corollary~\ref{corollary on TA}, the Atiyah class of the tangent dg vector bundle $(T_{A[1]}, L_{d_A})$ is related via an isomorphism to the Atiyah class of the dg vector bundle
\[
(\E:= \pi^\ast(E) = \pi^\ast(K[1] \oplus B), Q_\E:= d_{\CE} - \Omega).
\]
We summarize these into our main theorem.
\begin{theorem}\label{Theorem on Atiyah class of A[1]}
  For each choice of quadruple $(j,\tau,\nabla^K, \nabla^F)$, the Atiyah class $\At_{(A[1],d_A)}$ of the dg manifold $(A[1],d_A)$ arising from a regular Lie algebroid $A$ is related to the Atiyah class $\At_{(\E, Q_\E)}$ of the dg vector bundle $(\E, Q_\E)$ via the isomorphism
  \[
 \id_{T_{A[1]}^\ast} \otimes \Phi_{1,1} \colon H^1\left(\Gamma(T^\ast_{A[1]} \otimes \End(T_{A[1]}))\right) \xrightarrow{\cong} H^1\left(\Gamma(T_{A[1]}^\ast \otimes \End(\E))\right).
  \]
  Here $\Phi_{1,1}$ is the isomorphism for the integer pair $(1,1)$ in~\eqref{Eq: isom for mn tensor fields}.
  Moreover, the class
  \[
\begin{tikzcd}
  \left(\Phi_{1,0} \otimes \id_{\End(\E)}\right)\left(\At_{(\E,Q_\E)}\right) \in  H^1(\Gamma(\E^\ast \otimes \End(\E))) = H^1(A; E^\ast \otimes \End(E))
  \end{tikzcd}
   \]
   is represented by a (formal) sum $\alpha_A + \beta_A + \alpha_B + \beta_B \in \Omega_A(E^\ast \otimes \End(E))$ of cocycles,  where
 \begin{enumerate}
   \item $\alpha_A \in \Gamma(K^\ast[-1] \otimes \End(K[1]))$ is given by the Lie bracket $[-,-]_K$ on $\Gamma(K)$;
   \item $\beta_A \in \Omega^1_A(B^\ast \otimes \End(K[1]))$ is defined by
       \begin{align*}
        \beta_A(b, a_K[1])(a_0) &= \Omega(b)(a_0,a_K) =  \left(R^{\nabla^K}(\rho_A(a_0), j(b))a_K\right)[1] \\
        &\quad + \left(\nabla^K_{j(b)}[a_K, \pr_K(a_0)]_K - [\nabla_{j(b)}^K a_K, \pr_K(a_0)]_K - [a_K, \nabla^K_{j(b)}\pr_K(a_0)]_K\right)[1],
       \end{align*}
       for all $a_0 \in \Gamma(A)$ and $b \in \Gamma(B)$; Here $R^{\nabla^K}$ is the curvature of the linear connection $\nabla^K$ on $K$;
   \item $\alpha_B \in \Omega^1_A(B^\ast \otimes \End(B))$ is given by the $(1,1)$-component of the curvature $R^{\nabla^B}_{1,1}$ of a linear connection $\nabla^B$ on $B$ extending the Bott $F$-connection, i.e.,
  \begin{align*}
    \alpha_B(b, b^\prime)(a) &= R^{\nabla^B}(\rho_A(a), b)b^\prime = [\nabla^B_{\rho_A(a)}, \nabla^B_{j(b)}] b^\prime - \nabla^B_{[\rho_A(a),j(b)]}b^\prime,
  \end{align*}
   for all $a \in \Gamma(A)$ and $b, b^\prime \in \Gamma(B)$;
  \item  $\beta_B \in  \Omega_A^2(B^\ast \otimes B^\ast \otimes K[1])$ is defined by
  \begin{align*}
    \beta_B&(b_1, b_2)(a_0,a_1) = \nabla_{j(b_1)}(\Omega)(b_2; a_0, a_1) \\
    &:= \nabla^K_{j(b_1)} \Omega(b_2)(a_0, a_1) - \Omega(\nabla^B_{j(b_1)}b_2)(a_0,a_1) - \Omega(b_2)(a_0, \nabla^A_{j(b_1)}a_1) - \Omega(b_2)(\nabla^A_{j(b_1)}a_0, a_1),
  \end{align*}
     for all $a_0,a_1 \in \Gamma(A)$, $b_1,b_2 \in \Gamma(B)$.
\end{enumerate}
These four components satisfy
\begin{align*}
  d_{\CE}(\alpha_A) &= 0, & d_{\CE}(\alpha_B) &= 0, & \Omega(\beta_A) & =0, \\
  \Omega(\beta_B) &= 0, & \Omega(\alpha_A) &= d_{\CE}(\beta_A), &  \Omega(\alpha_B) &= d_{\CE}(\beta_B).
\end{align*}
\end{theorem}
To compute this Atiyah class, we need a $T_{A[1]}$-connection on $\E$.
For this purpose, we first choose a linear connection $\nabla^{K[1] \oplus B}$ on the Whitney sum $K[1] \oplus B$ defined by
\[
 \nabla^{K[1] \oplus B}_u (a_K[1] + b):= (\nabla^K_u a_K)[1] + \nabla^B_u b,
\]
for all $u \in \Gamma(T_M), a_K \in \Gamma(K)$ and $b \in \Gamma(B)$, where $\nabla^K$ is the chosen linear connection on $K$, and $\nabla^B$ is a linear connection $B$ extending the Bott $F$-connection.
We denote by $\nabla^\E$ the pullback connection of $\nabla^{K[1] \oplus B}$ on $\E$ along the projection $\pi \colon A[1] \to M$, that is,
\begin{equation}\label{Eq: pullback connection}
  \nabla^\E_{\mathcal{X}} (\alpha \otimes (a_K[1]+b)) = \mathcal{X}(\alpha) \otimes (a_K[1] + b) + (-1)^{\abs{\mathcal{X}}\abs{\alpha}}\alpha \otimes \nabla^{K[1] \oplus B}_{\pi_\ast(\mathcal{X})} (a_{K}[1] + b),
\end{equation}
for all homogeneous $\mathcal{X} \in \Gamma(T_{A[1]}), \alpha \in \Omega_A$ and all $a_K \in \Gamma(K), b \in \Gamma(B)$.
\begin{lemma}\label{Lem: Atiyah cocycle}
  Under the quasi-isomorphism
  \[
   \Psi \colon (\Omega_A(K[1] \oplus B), Q_\E) \to (\Gamma(T_{A[1]}), L_{d_A})
  \]
  defined in Corollary~\ref{corollary on TA}, the Atiyah cocycle $\At^{\nabla^\E}_{(\E, Q_\E)}$ of $(\E,Q_\E)$ with respect to the pullback connection $\nabla^\E$ in~\eqref{Eq: pullback connection} is given by
 \begin{align*}
   \At^{\nabla^\E}_{(\E,Q_\E)}(\Psi(a_K[1]), a_K^\prime[1]) &= [a_K, a_K^\prime]_K[1], \\
   \At_{(\E,Q_\E)}^{\nabla^\E}(\Psi(a_K[1]), b_1) &= \At_{(\E,Q_\E)}^{\nabla^\E}(\Psi(b_1), a_K[1]) = \Omega(b_1)(-,a_K) = -\left(R^{\bas}_{\nabla^A}(-,a_K)(j(b_1))\right)[1], \\
   \At_{(\E,Q_\E)}^{\nabla^\E}(\Psi(b_1), b_2) &= \alpha_B(b_1,b_2) + \beta_B(b_1,b_2),
 \end{align*}
 for all $a_K, a_K^\prime \in \Gamma(K)$ and $b_1,b_2 \in \Gamma(B)$, where $\alpha_B(b_1,b_2) \in \Omega_A^1(B)$ and $\beta_B(b_1,b_2) \in \Omega_A^2(K[1])$ are given, respectively, by
  \begin{align*}
    &\alpha_B(b_1,b_2)(a_0) = R^{\nabla^B}(\rho_A(a_0), j(b_1))b_2 =  \nabla^B_{\rho_A(a_0)}\nabla^B_{j(b_1)}b_2 - \nabla^B_{j(b_1)}\nabla^B_{\rho_A(a_0)}b_2 - \nabla^B_{[\rho_A(a_0), j(b_1)]}b_2, \\
    &\beta_B(b_1,b_2)(a_0,a_1) = \nabla_{j(b_1)}(\Omega)(b_2; a_0,a_1) \\
    &:= \nabla^K_{j(b_1)} \Omega(b_2)(a_0,a_1) - \Omega(\nabla^B_{j(b_1)}b_2)(a_0,a_1) - \Omega(b_2)(a_0, \nabla^A_{j(b_1)}a_1) - \Omega(b_2)(\nabla^A_{j(b_1)}a_0, a_1),
  \end{align*}
  for all $a_0, a_1 \in \Gamma(A)$ and $b_1, b_2 \in \Gamma(B)$.
\end{lemma}
\begin{proof}
Using the definition of the Atiyah cocycle, we compute case by case:
For all $a_K, a_K^\prime \in \Gamma(K)$, we have
   \begin{align*}
     &\quad \At_{(\E,Q_\E)}^{\nabla^\E}(\Psi(a_K[1]), a_K^\prime[1]) \\
    &= Q_\E(\nabla^\E_{\Psi(a_K[1])} a_K^\prime[1]) - \nabla^\E_{[d_A,\Psi(a_K[1])]}a_K^\prime[1] + \nabla^\E_{\Psi(a_K[1])}Q_\E(a_K^\prime[1]) \quad(\text{since $\Psi$ is a cochain map})\\
    &= Q_\E(\nabla^\E_{\Psi(a_K[1])} a_K^\prime[1]) - \nabla^\E_{\Psi(Q_\E(a_K[1]))}a_K^\prime[1] + \nabla^\E_{\Psi(a_K[1])}Q_\E(a_K^\prime[1]) \quad(\text{by Eqs.~\eqref{Eq: Def of Psi} and~\eqref{Eq: pullback connection}})\\
    &= Q_\E((\nabla^K_{\pi_\ast(I(a_K[1]))} a_K^\prime)[1]) - \nabla^K_{\pi_\ast (I(d_{\CE}a_K[1]))} a_K^\prime[1] + \nabla^\E_{I(a_K[1])}d_{\CE}(a_K^\prime)[1] \quad (\text{Since $\pi_\ast \circ I = 0$})\\
    &= \nabla^{\E}_{\iota_{a_K}} (d_{\CE}a_K^\prime)[1] = [a_K, a_K^\prime]_K[1].
   \end{align*}
   Similarly, for all $a_K \in \Gamma(K), b_1 \in \Gamma(B)$, we have
   \begin{align*}
   \At_{(\E,Q_\E)}^{\nabla^\E}(\Psi(a_K[1]), b_1) &= Q_\E(\nabla^\E_{\Psi(a_K[1])} b_1) - \nabla^\E_{[d_A,\Psi(a_K[1])]} b_1 + \nabla^\E_{\Psi(a_K[1])}Q_\E(b_1) \\
 &= \nabla^\E_{\iota_{a_K}} (d_{\CE}(b_1) - \Omega(b_1)) = -\iota_{a_K}\Omega(b_1) = \Omega(b_1)(-,a_K),
   \end{align*}
   and
   \begin{align*}
      \At_{(\E,Q_\E)}^{\nabla^\E}(\Psi(b_1), a_K[1]) &= Q_\E(\nabla^\E_{\Psi(b_1)} a_K[1]) - \nabla^\E_{[d_A,\Psi(b_1)]}a_K[1] -  \nabla^\E_{\Psi(b_1)}Q_\E(a_K[1]) \\
   &= Q_\E((\nabla^K_{\pi_\ast(\Psi(b_1))}a_K)[1]) - (\nabla^K_{\pi_\ast(\Psi(Q_\E(b_1)))}a_K)[1] - \nabla^\E_{\Psi(b_1)}Q_\E(a_K[1]) \\
   &= d_{\CE}((\nabla^K_{j(b_1)}a_K)[1]) - (\nabla^K_{j(d_{\CE}(b_1))}a_K)[1] - (\nabla^\E_{\nabla^A_{j(b_1)}}(d_{\CE}a_K)[1]),
   \end{align*}
thus, we have for all $a_0 \in \Gamma(A)$,
\begin{align*}
  &\qquad \At^{\nabla^\E}_{(\E,Q_\E)}(\Psi(b_1), a_K[1])(a_0) \\
  &= \left([a_0, \nabla^K_{j(b_1)}a_K]_A - \nabla^K_{\pr_B[\rho_A(a_0),j(b_1)]}a_K- \nabla^K_{j(b_1)}[a_0, a_K]_A + [\nabla^A_{j(b)}a_0, a_K]_A\right)[1] \\
  &=  \left([a_0, \nabla^K_{j(b_1)}a_K]_A - \nabla^K_{\pr_B[\rho_A(a_0),j(b_1)]}a_K - \nabla^K_{j(b_1)}[a_0, a_K]_A\right)[1] \\
  &\qquad+ [\nabla_{j(b_1)}^K(\pr_K(a_0)), a_K]_A[1] + (\nabla^K_{\pr_F[j(b_1), \rho_A(a_0)]}a_K)[1] \\
  &=  \left([\pr_K(a_0), \nabla^K_{j(b_1)}a_K]_K + [\nabla_{j(b_1)}^K(\pr_K(a_0)), a_K]_A - \nabla^K_{j(b_1)}[\pr_K(a_0), a_K]_A\right)[1] \\
  &\quad + \left(\nabla^K_{\rho_A(a_0)}\nabla^K_{j(b_1)}a_K - \nabla^K_{j(b_1)}\nabla^K_{\rho_A(a_0)}a_K + \nabla^K_{[j(b_1), \rho_A(a_0)]}a_K\right)[1] \\
  &=-R^{\bas}_{\nabla^A}(a_0, a_K)(j(b_1))[1] = \Omega(b_1)(a_0,a_K),
\end{align*}
which implies that
\[
\At^{\nabla^\E}_{(\E,Q_\E)}(\Psi(b_1), a_K[1]) = \Omega(b_1)(-, a_K) = \At^{\nabla^\E}_{(\E,Q_\E)}(\Psi(a_K[1]), b_1).
\]
Finally, for all $b_1, b_2 \in \Gamma(B)$, we have
\begin{align*}
   \At^{\nabla^\E}_{(\E,Q_\E)}(\Psi(b_1), b_2) &= Q_\E(\nabla^\E_{\Psi(b_1)}b_2) - \nabla^\E_{[d_A, \Psi(b_1)]}b_2 - \nabla^\E_{\Psi(b_1)}Q_\E(b_2) \quad \left(\text{since $\Psi$ is a cochain map}\right)\\
  &= Q_\E(\nabla^\E_{\Psi(b_1)}b_2) - \nabla^\E_{\Psi(Q_\E(b_1))}b_2 - \nabla^\E_{\Psi(b_1)}Q_\E(b_2) \quad \left(\text{by Eqs.~\eqref{Eq: Def of Psi} and~\eqref{Eq: pullback connection}}\right) \\
  &= Q_{\E}(\nabla^B_{j(b_1)}b_2) - \nabla^B_{d_{\CE}(b_1)}b_2 - \nabla^\E_{\nabla^A_{j(b_1)}}Q_{\E}(b_2) \\
  &= d_{\CE}(\nabla^B_{j(b_1)}b_2) - \nabla^B_{d_{\CE}(b_1)}b_2 - \nabla^\E_{\nabla^A_{j(b_1)}}d_{\CE}(b_2) + \nabla^\E_{\nabla^A_{j(b_1)}}\Omega(b_2) - \Omega(\nabla^B_{j(b_1)}b_2).
\end{align*}
These five terms are separated into the following two parts:
\begin{align*}
  \alpha_B(b_1, b_2) &:= d_{\CE}(\nabla^B_{j(b_1)}b_2) - \nabla^B_{d_{\CE}(b_1)}b_2 - \nabla^\E_{\nabla^A_{j(b_1)}}d_{\CE}(b_2) \in \Omega_A^1(B), \\
  \beta_B(b_1,b_2) &:= \nabla^\E_{\nabla^A_{j(b_1)}}\Omega(b_2) - \Omega(\nabla^B_{j(b_1)}b_2) \in \Omega_A^2(K[1]),
\end{align*}
satisfying for all $a_0, a_1 \in \Gamma(A)$,
\begin{align*}
  &\quad \alpha_B(b_1,b_2)(a_0) \\
  &= \nabla^B_{\rho_A(a_0)}\nabla^B_{j(b_1)}b_2 - \nabla^B_{\pr_B[\rho_A(a_0), j(b_1)]} b_2 \\
  &\qquad\qquad - \left( \nabla^\E_{\nabla^A_{j(b_1)}}\nabla^B_{\rho_A(a_0)}b_2 - \nabla^B_{\rho_A(\nabla^A_{j(b_1)}a_0)} b_2\right)  \qquad \left(\text{by Eqs.~\eqref{Eq: nablaA} and~\eqref{Eq: pullback connection}}\right)\\
  &= \nabla^B_{\rho_A(a_0)}\nabla^B_{j(b_1)}b_2 - \nabla^B_{\pr_B[\rho_A(a_0), j(b_1)]} b_2 - \left( \nabla^B_{j(b_1)}\nabla^B_{\rho_A(a_0)}b_2 - \nabla^B_{\pr_F[j(b_1),\rho_A(a_0)]} b_2\right) \\
  &= \nabla^B_{\rho_A(a_0)}\nabla^B_{j(b_1)}b_2 - \nabla^B_{[\rho_A(a_0), j(b_1)]} b_2 -  \nabla^B_{j(b_1)}\nabla^B_{\rho_A(a_0)}b_2 \\
  &= R^{\nabla^B}(\rho_A(a_0), j(b_1))b_2,
\end{align*}
and
\begin{align*}
  \beta_B(b_1,b_2)(a_0,a_1) &= (\nabla^\E_{\nabla^A_{j(b_1)}}\Omega(b_2))(a_0,a_1) - \Omega(\nabla^B_{j(b_1)}b_2)(a_0,a_1) \\
  &= \nabla^{K[1]}_{j(b_1)} \Omega(b_2)(a_0,a_1) - \Omega(b_2)(a_0, \nabla^A_{j(b_1)}a_1) - \Omega(b_2)(\nabla^A_{j(b_1)}a_0, a_1) \\
  &\qquad  - \Omega(\nabla^B_{j(b_1)}b_2)(a_0,a_1).
\end{align*}
\end{proof}
Theorem~\ref{Theorem on Atiyah class of A[1]} is an immediate consequence of Proposition~\ref{prop: invariance of Atiyah class} and Lemma~\ref{Lem: Atiyah cocycle}.

\subsection{Examples}
\subsubsection{Bundles of Lie algebras}
Recall that a bundle $K$ of Lie algebras over $M$ is a regular Lie algebroid $(K, [-,-]_K)$ with zero anchor.
For each linear connection $\nabla^K$ on $K$, the isomorphism in~\eqref{Eq: TA as representation UTH} becomes
\begin{equation}\label{Eq: isom of TK1}
  (\Gamma(T_{K[1]}), L_{d_K}) \xrightarrow{\cong} \left(\Omega_K(K[1] \oplus T_M), D_{\nabla^K} = \left(
                                                                                                   \begin{array}{cc}
                                                                                                     d^K_{\CE} & -\Omega \\
                                                                                                     0 & 0 \\
                                                                                                   \end{array}
                                                                                                 \right)
  \right),
\end{equation}
where $d^K_{\CE} \colon \Omega^\bullet_K(K[1]) \to \Omega_K^{\bullet+1}(K[1])$ is the Chevalley-Eilenberg differential of the $K$-module $K[1]$, and $\Omega$ is given by the basic curvature $R^{\bas}_{\nabla^K} \in \Omega_K^2(\Hom(T_M, K[1]))$ that is defined by
\[
 \Omega(u)(x,y) = -\left(R^{\bas}_{\nabla^K}(x,y)u\right)[1] := -\left(\nabla^K_u[x,y]_K - [\nabla^K_u x, y]_K - [x, \nabla^K_u y]_K\right)[1],
\]
for all $x, y \in \Gamma(K)$ and $u \in \Gamma(T_M)$.
It is clear that the basic curvature measures the compatibility between $\nabla^K$ and the Lie bracket $[-,-]_K$.

As an immediate consequence of Theorem~\ref{Theorem on Atiyah class of A[1]}, we obtain the following description on the Atiyah class of a bundle of Lie algebras.
\begin{proposition}\label{prop: Atiyah class of bundle of Lie algebras}
  Let $(K, [-,-]_K)$ be a bundle of Lie algebras over $M$. Under the isomorphism of cohomology induced from~\eqref{Eq: isom of TK1}
  \[
    H^1(\Gamma(T_{K[1]}^\ast \otimes \End(T_{K[1]})), L_{d_K}) \xrightarrow{\quad\cong\quad} H^1(K, (T^\ast_M \oplus K^\ast[-1]) \otimes \End(K[1] \oplus T_M)),
  \]
  where the right-hand side denotes the cohomology of the representation up to homotopy of $K$ on the tensor product $(T^\ast_M \oplus K^\ast[-1]) \otimes \End(K[1] \oplus T_M)$ of graded vector bundles over $M$. Then the Atiyah class $\At_{(K[1],d_K)}$ of the associated dg manifold $(K[1], d_K)$ is represented by the sum of three terms
  \[
   \alpha_K + \beta_{K} + \beta_M,
  \]
  where
  \begin{enumerate}
    \item $\alpha_K \in \Gamma(K^\ast[-1] \otimes \End(K[1]))$ is given by the Lie bracket $[-,-]_K$;
    \item $\beta_K \in \Omega_K^1(T^\ast_M \otimes \End(K[1]))$ is defined by
        \[
         \beta_K(u, x[1])(y) := \left(\nabla^K_u[x,y]_K - [\nabla^K_u x, y]_K - [x, \nabla^K_u y]_K\right)[1],
        \]
        for all $u \in \Gamma(T_M), x,y \in \Gamma(K)$;  Here $\nabla^K$ is a linear connection on $K$;
    \item $\beta_M \in \Omega_K^2(T^\ast_M \otimes T^\ast_M \otimes K[1])$ is defined by
    \begin{align*}
    \beta_M(u,v)(x,y) &= -\nabla^{K[1]}_u\beta_K(v,x[1])(y) + \beta_K(\nabla^M_u v, x[1])(y) + \beta_K(v, (\nabla^K_ux)[1])(y) \\
    &\quad + \beta_K(v, x[1])(\nabla^K_u y),
   \end{align*}
   for all $u,v\in \Gamma(T_M)$ and $x,y \in \Gamma(K)$.  Here $\nabla^K$ is a linear connection on $K$ and $\nabla^M$ is an affine connection on $M$.
  \end{enumerate}
\end{proposition}
In particular, if $(K, [-,-]_K)$ is a Lie algebra bundle, that is, the fiber Lie algebra is fixed in local trivialization, then there exists a linear connection $\nabla^K$ on $K$ such that its basic curvature $R^{\bas}_{\nabla^K}$ vanishes (see Proposition 2.13 in~\cite{AC}). Thus, we obtain the following corollary.
\begin{corollary}\label{corollary on Lie algebra bundle}
 Let $(K,[-,-]_K)$ be a Lie algebra bundle. Then the Atiyah class $\At_{(K[1],d_K)}$ of the associated dg manifold $(K[1],d_K)$ is represented by the Lie bracket $[-,-]_K$, viewed as a degree $1$ element in $\Gamma(K^\ast[-1] \otimes \End(K[1]))$.
\end{corollary}

\subsubsection{Integrable distributions}
Consider a regular Lie algebroid $A$ whose anchor map $\rho$ is injective. In this case, $A$ is identified with its characteristic distribution $F: =\rho_A(A) \subseteq T_M$, the tangent bundle of a regular foliation $\mathcal{F} \subseteq M$.
Note that $(T_M, F)$ is a Lie algebroid pair (or Lie pair for short) over $M$.
We briefly recall from~\cite{CSX} the Atiyah class of the Lie pair $(T_M, F)$.
For each splitting $j \colon B = T_M/F \to T_M$ of the short exact sequence~\eqref{SES} and a linear connection $\nabla^B$ on $B$ extending the Bott $F$-connection, there is a Chevalley-Eilenberg $1$-cocycle of $F$
\[
  \At_B^{\nabla^B} \in C^1(F; B^\ast \otimes \End(B)) = \Gamma(F^\ast \otimes B^\ast \otimes \End(B)),
\]
defined by
\[
  \At_B^{\nabla^B}(u_F, b_1)b_2 := R^{\nabla^B}(u_F, j(b_1))b_2 = \nabla^B_{u_F}\nabla^B_{j(b_1)} b_2- \nabla^B_{j(b_1)}\nabla^B_{u_F} b_2 - \nabla^B_{[u_F, j(b_1)]}b_2,
\]
for all $u_F \in \Gamma(F)$ and $b_1, b_2 \in \Gamma(B)$. The cohomology class
\[
 \At_B = \left[\At_B^{\nabla^B}\right] \in H^1_{\CE}(F; B^\ast \otimes \End(B))
\]
does not depend on the choice of $j$ and $\nabla^B$, and is called the Atiyah class of the Lie pair $(T_M, F)$, which is also known as the Molino class~\cite{Molino} of the foliation $\mathcal{F}$ induced from $F$.
Applying Theorem~\ref{Theorem on Atiyah class of A[1]} to this case, we recover the following result on Atiyah classes of integrable distributions.
\begin{proposition}[\cite{CXX}]\label{prop: Atiyah class of foliation}
    Let $F \subseteq T_M$ be an integrable distribution, i.e., the tangent bundle of some regular foliation in $M$. Then the isomorphism $\Phi_{2,1}$ for the integer pair $(2,1)$ in~\eqref{Eq: isom for mn tensor fields}, which now becomes
    \[
     \Phi_{2,1} \colon H^1(\Gamma(T^\ast_{F[1]} \otimes \End(T_{F[1]}))) \xrightarrow{\quad \cong \quad} H^1(\Omega_F(B^\ast \otimes \End(B)), d_{\CE}) = H^1_{\CE}(F; B^\ast \otimes \End(B)),
    \]
    sends the Atiyah class $\At_{(F[1],d_F)}$ of the dg manifold $(F[1], d_F)$ to the Atiyah class $\At_{B}$ of the Lie pair $(T_M, F)$.
\end{proposition}

\subsection{Functoriality with respect to Atiyah sequence}
Given a regular Lie algebroid $(A, \rho_A, [-,-]_A)$ with the Atiyah sequence
\begin{equation}\label{Eq: Atiyah sequence}
   0 \to K = \ker\rho_A \xrightarrow{i} A \xrightarrow{\rho_A} F \to 0,
\end{equation}
there are three Atiyah classes $\At_{(K[1],d_K)}, \At_{(A[1],d_A)}$ and $\At_{(F[1],d_F)}$.
we now study their relationship.

By Theorem~\ref{Theorem on Atiyah class of A[1]}, the Atiyah class $\At_{(A[1],d_A)}$ is related to  $\At_{(\E, Q_\E)}$ via the isomorphism
\[
  \At_{(A[1],d_A)} \in H^1(\Gamma(T_{A[1]}^\ast \otimes \End(T_{A[1]}))) \cong  H^1(A; E^\ast \otimes \End(E)) \cong H^1(\Gamma(T_{A[1]}^\ast \otimes \End(\E))) \ni \At_{(\E, Q_\E)}.
\]
Note that the inclusion $i$ in Sequence~\eqref{Eq: Atiyah sequence} induces an inclusion $i \colon (K[1],d_K) \hookrightarrow (A[1],d_A)$ of dg manifolds.
The restriction of the homological vector field $Q_\E$ onto $\E\!\mid_{K[1]} = i^\ast \E$
 \[
   Q_\E\!\mid_{K[1]}  = \left(
           \begin{array}{cc}
             d^K_{\CE} & - i^\ast \Omega \\
             0 & 0 \\
           \end{array}
         \right) \colon \Gamma(\E\!\mid_{K[1]}) = \Omega_K(K[1] \oplus B) \to \Omega_K(K[1] \oplus B)[1]
 \]
makes $(\E\!\!\mid_{K[1]}, Q_\E\!\!\mid_{K[1]})$ into a dg vector bundle over the dg manifold $(K[1],d_K)$.
Here $d_{\CE}^K$ is the Chevalley-Eilenberg differential of the adjoint $K$-module $K[1]$, and $i^\ast \Omega \in \Omega_K^2(B^\ast \otimes K[1])$ is the pullback of $\Omega$ along the inclusion $i$.

\begin{theorem}\label{Thm:Funtoriality}
    Let $A$ be a regular Lie algebroid with Atiyah Sequence~\eqref{Eq: Atiyah sequence}.
    \begin{enumerate}
    \item The standard projection map
      \[
         \operatorname{Pr} \colon H^1(A; E^\ast \otimes \End(E)) \to H^1_{\CE}(A; B^\ast \otimes \End(B))
      \]
      sends the Atiyah class $\At_{(\E,Q_\E)} \in  H^1(\Gamma(T_{A[1]}^\ast \otimes \End(\E)))  \cong H^1(A; E^\ast \otimes \End(E))$ of the dg vector bundle $(\E,Q_\E)$, which is related to the Atiyah class $\At_{(A[1],d_A)}$ of $A$ via an isomorphism according to Theorem~\ref{Theorem on Atiyah class of A[1]}, to the pullback class $\rho_A^\ast(\At_B)$ of the Atiyah class $\At_B \in H^1_{\CE}(F; B^\ast \otimes \End(B))$ of the Lie pair $(T_M, F)$, the latter of which is related to the Atiyah class $\At_{(F[1],d_F)}$ of the dg manifold $(F[1],d_F)$ via an isomorphism according to Proposition~\ref{prop: Atiyah class of foliation}.
      \item For each splitting $j \colon B \to T_M$ of the short exact sequence~\eqref{SES}, there exists an inclusion
     \[
        H^1(K; E^\ast \otimes \End(E)) \hookrightarrow H^1(K, (T_M^\ast \oplus K^\ast[-1]) \otimes \End(K[1] \oplus T_M)) \cong  H^1(\Gamma(T_{K[1]}^\ast \otimes \End(T_{K[1]}))).
     \]
      The Atiyah class $\At_{(K[1],d_K)} $ of the bundle $K$ of Lie algebras, which lives in $H^1(K; E^\ast \otimes \End(E))$, is equal to the Atiyah class of the dg vector bundle $(\E\!\mid_{K[1]}, Q_\E\!\mid_{K[1]})$ over $(K[1],d_K)$:
      \begin{align*}
        \At_{(\E \mid_{K[1]}, Q_\E \mid_{K[1]})} \in H^1(K; E^\ast \otimes \End(E)) &\subset H^1(K;  (T_M^\ast \otimes K^\ast[-1]) \otimes \End(E)) \\
        &\cong H^1(\Gamma(T_{K[1]}^\ast \otimes \End(\E\!\mid_{K[1]}))).
      \end{align*}
  Therefore, the restriction map
  \[
      i^\ast \colon H^1(A; E^\ast \otimes \End(E)) \to H^1(K; E^\ast \otimes \End(E))\hookrightarrow  H^1(\Gamma(T_{K[1]}^\ast \otimes \End(T_{K[1]}))).
  \]
  sends the Atiyah class $\At_{(\E,Q_\E)} \in  H^1(\Gamma(T_{A[1]}^\ast \otimes \End(\E)))  \cong H^1(A; E^\ast \otimes \End(E))$ of $(\E, Q_\E)$, to the Atiyah class $\At_{(K[1],d_K)}$.
    \end{enumerate}
\end{theorem}
\begin{proof}
The first statement follows immediately from Theorem~\ref{Theorem on Atiyah class of A[1]}.
For the second one, let us fix the quadruple $(j,\tau,\nabla^K, \nabla^F)$ as in~\eqref{Eq: quadruple} and choose an affine connection $\nabla^M$ satisfying
\begin{align*}
  \nabla^M_{u_F} j(b) &= \pr_B [u_F, j(b)], & \nabla^M_u u_F &= \nabla^F_u u_F,
\end{align*}
for all $u_F \in \Gamma(F), u \in \Gamma(T_M)$, and $b \in \Gamma(B)$.
The dual map $\pr_B^\ast \colon B^\ast \to T_M^\ast$ of the projection $\pr_B \colon T_M \hookrightarrow B$ and the chosen splitting $j \colon B \to T_M$ induce an inclusion of the cohomology spaces
\[
  H^1(K; E^\ast \otimes \End(E)) \hookrightarrow H^1(K; (T_M^\ast \otimes K^\ast[-1]) \otimes \End(K[1] \oplus T_M)).
\]
By Proposition~\ref{prop: Atiyah class of bundle of Lie algebras}, the Atiyah class $\At_{(K[1],d_K)}$ is represented by a formal sum
\[
 \alpha_K + \beta_{K} + \beta_M.
\]
Here $\alpha_K$ is the Lie bracket $[-,-]_K$ on $\Gamma(K)$, and the element $\beta_{K} \in \Omega_K^1(T^\ast_M \otimes \End(K[1]))$ satisfies for all $x, y \in \Gamma(K)$ and $u_F \in \Gamma(F)$,
\begin{align}\label{Eq: beta FKK = 0}
 \beta_{K}(u_F, x[1]) (y) &= \left(\nabla^K_{u_F}[x,y]_K - [\nabla^K_{u_F} x, y]_K - [x, \nabla^K_{u_F} y]_K\right)[1] \notag \\
                                              &= \left([\tau(u_F), [x,y]_K]_A - [[\tau(u_F), x]_A, y]_K - [x, [\tau(u_F), y]_A]_K\right)[1] \notag \\
                                              &= 0.
\end{align}
Thus, we have
\[
 \beta_K = i^\ast\beta_A \in \Omega^1(K; B^\ast \otimes \End(K[1])) \subset \Omega^1(K; T^\ast_M \otimes \End(K[1])).
\]
Meanwhile, by Equation~\eqref{Eq: beta FKK = 0}, the element $\beta_M \in \Omega_K^2(T^\ast_M \otimes T^\ast_M \otimes K[1])$ satisfies
\begin{align*}
  &\quad \beta_M(u, v_F)(x,y) \\
  &= \nabla^{K[1]}_{u}(\beta_K(v_F; x[1])(y)) - \beta_K(\nabla^M_{u}v_F, x[1])(y) - \beta_K(v_F, (\nabla^K_{u}x)[1])(y) - \beta_K(v_F, x[1])(\nabla^K_{u}y)\\
  & = 0,
\end{align*}
for all $u \in \Gamma(T_M), v_F \in \Gamma(F), x, y \in \Gamma(K)$, and
\begin{align*}
\beta_M(u_F, j(b))(x, y) &=  \nabla^{K[1]}_{u_F}(\beta^{\nabla^K}_K(j(b), x[1])(y)) \\
&\quad- \beta^{\nabla^K}(\nabla^M_{u_F}j(b), x[1])(y) - \beta^{\nabla^K}(j(b), (\nabla^K_{u_F}x)[1])(y) - \beta^{\nabla^K}(j(b), x[1]) (\nabla^K_{u_F}y) \\
&= \left(R^{\nabla^K}(u_F, j(b))[x,y]_K - [R^{\nabla^K}(u_F,j(b))x, y]_K - [x, R^{\nabla^K}(u_F,j(b))y]_K\right)[1] \\
&= (d_{\CE}^K(R^{\nabla^K})(u_F, j(b))(x,y))[1],
\end{align*}
for all $u_F \in \Gamma(F), b \in \Gamma(B)$ and $x,y \in \Gamma(K)$. Thus, we also have
\[
  \beta_M = i^\ast \beta_B \in H^1(K; B^\ast \otimes B^\ast \otimes K[1]) \subset H^1(K; T^\ast_M \otimes T^\ast_M \otimes K[1]).
\]
Hence, applying Theorem~\ref{Theorem on Atiyah class of A[1]}, we obtain
\begin{align*}
   \At_{(K[1],d_K)} &= \alpha_K + \beta_K + \beta_M = i^\ast(\alpha_A + \beta_A + \beta_B) \\
   &= i^\ast(\At_{(\E,Q_\E)}) = \At_{(\E\mid_{K[1]}, Q_\E\mid_{K[1]})}.
\end{align*}
\end{proof}
As an immediate consequence, we obtain the following vanishing result.
\begin{corollary}
  If the Atiyah class $\At_{(A[1],d_A)}$ of a regular Lie algebroid $A$ vanishes, then both the Atiyah class $\At_{(K[1],d_K)}$ of the bundle $K=\ker\rho_A$ of Lie algebras and the Atiyah class $\At_{(F[1],d_F)}$ of the characteristic distribution $F = \operatorname{Im}\rho_A$ vanish.
\end{corollary}

\section{Scalar Atiyah and Todd classes}
\subsection{Scalar Atiyah classes}
Now we study the scalar Atiyah classes of the dg manifold $(A[1], d_A)$.
As the first step, by taking tensor product on the isomorphism in Corollary~\ref{corollary on TastA}, we obtain an isomorphism between the cohomology of differential $k$-forms on the dg manifold $(A[1], d_A)$ of total degree $k$ and the cohomology of representation up to homotopy of $A$ on the graded vector bundle $\wedge^k E^\ast \cong \wedge^k (K^\ast[-1] \oplus B^\ast)$ of total degree $k$ , i.e.,
\begin{equation}\label{Eq: isom for kforms}
  H^k(\Omega^k(A[1]), L_{d_A}) \cong H^k(A; \wedge^k E^\ast) = \bigoplus_{q=0}^k H^k(A; \wedge^{k-q} B^\ast \otimes (S^q K^\ast)[-q]).
\end{equation}
Recall that the differential on $\Omega_A(\wedge^kE^\ast)$ is induced by Leibniz rule from the differential $Q_{\E^\ast} = d_{\CE} + \Omega^\ast$ on $\E^\ast$, where $\Omega^\ast \in \Omega_A^2(\Hom(K^\ast[-1], B^\ast))$ is the $\Omega_A$-linear dual of $\Omega \in \Omega_A^2(\Hom(B,K[1]))$.
The projection onto the $(q=0)$-component defines a cochain map
\[
 \operatorname{Pr}\colon \left(\bigoplus_{q=0}^k \Omega_A(\wedge^{k-q} B^\ast \otimes (S^q K^\ast)[-q]), Q_{\E^\ast}\right) \to (\Omega_A(\wedge^k B^\ast), d_{\CE}),
\]
thus induces a projection from the cohomology $H^k(A; \wedge^k E^\ast)$ of the representation up to homotopy of $A$ on the graded vector bundle $\wedge^k E^\ast$ of total degree $k$ to the $k$-th Chevalley-Eilenberg cohomology $H^k_{\CE}(A; \wedge^k B^\ast)$ of the $A$-module $\wedge^k B^\ast$
\begin{equation}\label{Eq: projection on cohomology of kforms}
  \operatorname{Pr} \colon H^k(A; \wedge^k E^\ast) \to H^k_{\CE}(A; \wedge^k B^\ast).
\end{equation}
Note that each homogeneous element $T \in \Omega_A(E^\ast \otimes \End(E))$ has the matrix form
\[
\left(
  \begin{array}{cc}
    T_1 & T_2 \\
    T_3 & T_4 \\
  \end{array}
\right),
\]
where
\begin{align*}
  T_1 &\in \Omega_A(E^\ast \otimes \End(K[1])), & T_2 &\in \Omega_A(E^\ast \otimes \Hom(B,K[1])), \\
  T_3 &\in \Omega_A(E^\ast \otimes \Hom(K[1],B)), & T_4 &\in \Omega_A(E^\ast \otimes \End(B)).
\end{align*}
The supertrace $\str(T)$ of $T$ is defined by
\[
 \str(T):= \tr(T_4) - \tr(T_1) \in \Omega_A(E^\ast).
\]
With the help of Theorem~\ref{Theorem on Atiyah class of A[1]}, we obtain the following description on scalar Atiyah classes of $(A[1],d_A)$.
\begin{proposition}\label{prop on scaler atiyah classes}
  Let $A$ be a regular Lie algebroid. For each choice of quadruple $(j,\tau,\nabla^K, \nabla^F)$, under the isomorphism in Equation~\eqref{Eq: isom for kforms}, the $k$-th scalar Atiyah class $\ch_k(A[1],d_A)$ of the associated dg manifold $(A[1],d_A)$ is represented by
  \begin{align*}
  &\frac{1}{k!}\left(\frac{i}{2\pi}\right)^k\left(\tr\left(\alpha_B^k\right) - \sum_{q=0}^k\frac{k!}{q!(k-q)!} \tr\left(\alpha_A^q \beta^{k-q}_A\right)\right),
  \end{align*}
  where $\alpha_A \in \Gamma(K^\ast[-1] \otimes \End(K[1])), \beta_A\in \Omega_A^1(B^\ast \otimes \End(K[1]))$ and $\alpha_B \in \Omega_A^1(B^\ast \otimes \End(B))$ are elements defined in Theorem~\ref{Theorem on Atiyah class of A[1]}, satisfying
  \begin{align*}
   \tr\left(\alpha_B^k\right) &\in \Omega_A^k(\wedge^k B^\ast) \subset \Omega_A^k(\wedge^k E^\ast), \\
   \tr\left(\alpha_A^q\beta^{k-q}_A\right) &\in \Omega^{k-q}_A (\wedge^{k-q}B^\ast \otimes (S^qK^\ast)[-q]),
\end{align*}
and
\[
  \frac{k!}{q!(k-q)!}d_{\CE}\left(\tr\left(\alpha_A^q\beta^{k-q}_A\right)\right) = \frac{k!}{(q+1)!(k-q-1)!} \Omega\left( \tr\left(\alpha_A^{q+1}\beta^{k-q-1}_A\right)\right),
\]
for all $0 \leq q \leq k$.
\end{proposition}
\begin{proof}
  By Theorem~\ref{Theorem on Atiyah class of A[1]}, the $k$-th power the Atiyah class is represented by
  \[
    \left(
                          \begin{array}{cc}
                            (\alpha_A + \beta_A)^k & \ast \\
                            0 & \alpha_B^k \\
                          \end{array}
                        \right),
  \]
  where $\alpha_A \in \Gamma(K^\ast[-1] \otimes \End(K[1]))$ is given by the Lie bracket $[-,-]_K$ on $\Gamma(K)$, $\beta_A \in \Omega_A^1(B^\ast \otimes \End(K[1]))$, and $\alpha_B \in \Omega_A^1(B^\ast \otimes \End(B))$.
  Thus, we have
  \begin{align} \label{Eq: str of Atk}
  \str(\At^k_{(A[1],d_A)}) &= \tr\left(\alpha_B^k\right) - \tr\left((\alpha_A + \beta_A)^k\right) = \tr\left( \alpha_B^k \right) - \sum_{q=0}^k\frac{k!}{q!(k-q)!} \tr\left(\alpha^q_A\beta^{k-q}_A\right).
  \end{align}
Substituting into the definition of scalar Atiyah classes, we complete the proof.
\end{proof}
As an immediate application, we recall that the scalar Atiyah classes of the Lie pair $(T_M, F)$~\cite{CSX} is defined by
\[
 \ch_k(B):= \frac{1}{k!}\left(\frac{i}{2\pi}\right)^k \tr(\At_B^k),
\]
where $\At_B^k$ denotes the image of $\At_B^{\otimes k}$ under the natural map
\[
  H^1_{\CE}(F; B^\ast \otimes \End(B)) \times \cdots \times H^1_{\CE}(F; B^\ast \otimes \End(B)) \to H^k_{\CE}(F; \wedge^k B^\ast)
\]
induced by the composition in $\End(B)$ and the wedge product in $\wedge B^\ast$.
Hence, applying Proposition~\ref{prop: Atiyah class of foliation}  and Proposition~\ref{prop on scaler atiyah classes} to this particular regular Lie algebroid $F$, we recover the following result on scalar Atiyah classes of integrable distributions.
\begin{proposition}[\cite{CXX}]\label{prop: scalar Atiyah classes for integrable distributions}
  Let $F \subseteq T_M$ be an integrable distribution, i.e., the tangent bundle of some regular foliation in $M$. Then the isomorphism~\eqref{Eq: isom for kforms} becomes
  \[
    H^k(\Omega^k(F[1]), L_{d_F}) \cong H^k_{\CE}(F; \wedge^k B^\ast),
  \]
   which sends the scalar Atiyah classes $\ch_k(F[1],d_F)$ of the dg manifold $(F[1],d_F)$  to the scalar Atiyah classes $\ch_k(B)$ of the Lie pair $(T_M, F)$.
\end{proposition}
Finally, using Theorem~\ref{Thm:Funtoriality} and Proposition~\ref{prop on scaler atiyah classes}, we see that scalar Atiyah classes of dg manifolds arising from a regular Lie algebroid respect the associated Atiyah sequence.
\begin{proposition}\label{prop: functoriality on scalar Atiyah classes}
Let $A$ be a regular Lie algebroid with Atiyah Sequence~\eqref{Eq: Atiyah sequence}. For each positive integer $k$,
\begin{enumerate}
\item the projection map $\operatorname{Pr}$ in~\eqref{Eq: projection on cohomology of kforms} sends the $k$-th scalar Atiyah class of $A$
     \[
       \ch_k(A[1],d_A) \in H^k(\Omega^k(A[1]), L_{d_A}) \cong H^k(A; \wedge^k E^\ast)
     \]
     to the pullback class $\rho_A^\ast(\ch_k(B))$ of the $k$-th scalar Atiyah class $\ch_k(B)$ of the Lie pair $(T_M, F)$, which is related to the $k$-th scalar Atiyah class $\ch_k(F[1],d_F)$ of the integrable distribution $F$ by an isomorphism according to Proposition~\ref{prop: scalar Atiyah classes for integrable distributions}.
\item the restriction map
\[
  i^\ast \colon H^k(A; \wedge^k E^\ast) \to H^k(K; \wedge^k E^\ast) \subset H^k(\Omega^k(K[1]), L_{d_K})
\]
induced from the inclusion $i \colon K \hookrightarrow A$ sends the $k$-th scalar Atiyah class $\ch_k(A[1],d_A)$ of $A$ to the $k$-th scalar Atiyah class $\ch_k(K[1],d_K)$ of $K$.
\end{enumerate}
\end{proposition}

\subsection{Todd classes}
We now study the Todd class of the dg manifold $(A[1],d_A)$.
Let
\[
  P(x) = \frac{x}{1-e^{-x}} = \sum_{k \geq 0} \frac{(-1)^k}{k!}B_k x^k,
\]
where $B_k$ is the k-th Bernoulli number.

\begin{proposition}\label{prop on Todd class}
  Let $A$ be a regular Lie algebroid. For each choice of quadruple $(j,\tau,\nabla^K, \nabla^F)$, under the isomorphism
  \[
   \prod_{k\geq 0}  H^k(\Omega^k(A[1]), L_{d_A}) \cong \prod_{k\geq 0} \bigoplus_{q=0}^k H^k(A; \wedge^{k-q} B^\ast \otimes (S^q K^\ast)[-q]),
  \]
  the Todd class $ \Td_{(A[1],d_A)} $ of the dg manifold $(A[1], d_A)$ arising from a regular Lie algebroid $A$ is represented by
  \begin{align*}
\det(P(\alpha_B))\det(P^{-1}(\alpha_A + \beta_A))  &=\det(P(\alpha_B))
  \exp\left(\sum_{k\geq 1}\sum_{q=0}^{k}\frac{B_k}{k}\frac{\tr\left( \alpha^q_A \beta^{k-q}_A \right)}{q!(k-q)!}\right),
  \end{align*}
  where $\alpha_A \in \Gamma(K^\ast[-1] \otimes \End(K[1])), \beta_A\in \Omega_A^1(B^\ast \otimes \End(K[1]))$ and $\alpha_B \in \Omega_A^1(B^\ast \otimes \End(B))$ are elements defined in Theorem~\ref{Theorem on Atiyah class of A[1]}.
\end{proposition}
\begin{proof}
Substituting the Atiyah cocycle $\At^\nabla:= \alpha_A + \beta_A+\alpha_B+\beta_B$ representing the Atiyah class $\At_{(\E,Q_\E)}$ in Theorem~\ref{Theorem on Atiyah class of A[1]} into the definition of Todd class, we see that the Todd class $\Td_{(A[1],d_A)}$ is indeed represented by
\begin{align*}
\Ber (P(\At^\nabla)) &=
   \Ber\left(P\left( \left(
                         \begin{array}{cc}
                           \alpha_A + \beta_A & \beta_A + \beta_B \\
                           0 & \alpha_B \\
                         \end{array}
                       \right)
  \right)\right) \\
  &= \Ber\left(\left(
                 \begin{array}{cc}
                   P(\alpha_A + \beta_A) & \ast \\
                   0 & P(\alpha_B) \\
                 \end{array}
               \right)
  \right) \\
  &= \det(P(\alpha_B))\det(P^{-1}(\alpha_A + \beta_A)).
\end{align*}
Note further that we can rewrite the polynomial $P(x)$ in the following way (cf.~\cite{HBJ}):
\[
P(x) = \exp\left(-\sum_{k \geq 1} \frac{B_k}{k}\frac{x^k}{k!}\right).
\]
It follows that
\begin{align*}
  \Ber (P(\At^\nabla))  &= \Ber\left(\exp\left(- \sum_{k \geq 1} \frac{B_k}{k}\frac{\At_{(A[1], d_A)}^k}{k!}\right)\right) \\
  &= \exp\left(- \sum_{k\geq 1}\frac{B_k}{k}\frac{\str\left(\At_{(A[1],d_A)}^k\right)}{k!} \right) \qquad (\text{by Eq.~\eqref{Eq: str of Atk}})\\
  &= \exp\left(- \sum_{k\geq 1}\frac{B_k}{k}\frac{\tr\left(\alpha_B^k\right)}{k!}\right)
  \exp\left(\sum_{k \geq 1}\sum_{q=0}^{k}\frac{B_k}{k}\frac{\tr\left( \alpha^q_A \beta^{k-q}_A \right)} {q!(k-q)!}\right) \\
 &= \det(P(\alpha_B))\exp\left(\sum_{k\geq 1} \sum_{q=0}^{k}\frac{B_k}{k} \frac{\tr\left( \alpha^q_A \beta^{k-q}_A \right)}{q!(k-q)!}\right).
\end{align*}
\end{proof}

Recall that the Todd class~\cite{CSX} of the Lie pair $(T_M, F)$ is the cohomology class
\[
 \Td_B = \det(P(\At_B)) = \det \left(\frac{\At_B}{1 - e^{-\At_B}} \right) \in \bigoplus_{k \geq 0}H^k_\CE(F; \wedge^k B^\ast).
\]
Applying Theorem~\ref{Theorem on Atiyah class of A[1]} and Proposition~\ref{prop on Todd class}, we recover the following result on Todd classes of integrable distributions.
\begin{corollary}[\cite{CXX}]\label{prop: Todd class of integrable distribution}
  Let $F \subseteq T_M$ be an integrable distribution, i.e., the tangent bundle of some regular foliation in $M$. Then the isomorphism
  \[
   \prod_{k\geq 0} H^k(\Omega^k(F[1]), L_{d_F}) \cong \prod_{k\geq 0} H_{\CE}^k(F; \wedge^{k} B^\ast),
  \]
  sends the Todd classes $\Td_{(F[1],d_F)}$ of the dg manifold $(F[1],d_F)$  to the Todd classes $\Td_{B}$ of the Lie pair $(T_M, F)$.
\end{corollary}

Finally, applying Theorem~\ref{Thm:Funtoriality} and Proposition~\ref{prop on Todd class}, we see that Todd classes of dg manifolds arising from a regular Lie algebroid respect  the associated Atiyah sequence as well.
\begin{proposition}\label{prop: functoriality of Todd classes}
Let $A$ be a regular Lie algebroid with Atiyah Sequence~\eqref{Eq: Atiyah sequence}.
\begin{enumerate}
\item The projection map
\[
  \operatorname{Pr} \colon \prod_{k\geq 0} \bigoplus_{q=0}^k H^k(A; \wedge^{k-q} B^\ast \otimes (S^q K^\ast)[-q]) \to \prod_{k\geq 0} H^k_{\CE}(A; \wedge^k B^\ast)
\]
induced from~\eqref{Eq: projection on cohomology of kforms} sends the Todd class $\Td_{(A[1],d_A)}$ of $A$ to the pullback $\rho_A^\ast \Td_B$ of the Todd class $\Td_B$ of the Lie pair $(T_M, F)$, the latter of which is related to the Todd class $\Td_{(F[1],d_F)}$ of the integrable distribution $F$ via an isomorphism by Corollary~\ref{prop: Todd class of integrable distribution}.
\item Under the isomorphisms
\[
   \prod_{k\geq 0}  H^k(\Omega^k(A[1]), L_{d_A}) \cong \prod_{k\geq 0} \bigoplus_{q=0}^k H^k(A; \wedge^{k-q} B^\ast \otimes (S^q K^\ast)[-q]),
\]
and
 \[
   \prod_{k\geq 0}  H^k(\Omega^k(K[1]), L_{d_K}) \cong \prod_{k\geq 0} \bigoplus_{q=0}^k H^k(A; \wedge^{k-q} T_M^\ast \otimes (S^q K^\ast)[-q]),
 \]
the restriction map
\[
  \prod_{k\geq 0} \bigoplus_{q=0}^k H^k(A; \wedge^{k-q} B^\ast \otimes (S^q K^\ast)[-q]) \to
  \prod_{k\geq 0} \bigoplus_{q=0}^k H^k(K; \wedge^{k-q} T_M^\ast \otimes (S^q K^\ast)[-q])
\]
induced from the inclusion $i \colon K \hookrightarrow A$ and the projection $\pr_B \colon T_M \to B$ sends the Todd class $\Td_{(A[1],d_A)}$ of $A$ to the Todd class $\Td_{(K[1],d_K)}$ of $K$ represented by
\begin{align*}
 \det(P^{-1}(\alpha_K+\beta_K)) &= \det(P^{-1}(\alpha_K))\exp\left(\sum_{k\geq 1}\sum_{q=0}^{k-1}\frac{B_k}{k}\frac{\tr\left( \alpha^q_K \beta^{k-q}_K \right)}{q!(k-q)!}\right),
\end{align*}
where
\[
\det(P^{-1}(\alpha_K)) = \det\left(\frac{1-e^{-\alpha_K}}{\alpha_K}\right) \in \prod_{k\geq0}\Gamma((S^kK^\ast)[-k])
\]
consists of Duflo elements of this bundle of Lie algebras.
\end{enumerate}
\end{proposition}

\subsection{Application to locally splittable cases}
Assume that $A$ is a locally splittable regular Lie algebroid, i.e., the characteristic class of $A$
\[
[\omega] \in H^2_{\CE}(A; \Hom(B,K))
\]
vanishes. According to {Proposition 7.2} in~\cite{GsM}, one can choose a quadruple $(\tau, j, \nabla^K, \nabla^F)$ such that $\Omega = 0$.
By Corollary~\ref{corollary on TA}, we obtain a contraction
  \[
    \begin{tikzcd}
	(T_{A[1]}, L_{d_A}) \arrow[loop left, distance=2em, start anchor={[yshift=-1ex]west}, end anchor={[yshift=1ex]west}]{}{H} \arrow[r,yshift = 0.7ex, "\Phi"] & (\E = \pi^\ast(K[1] \oplus B), Q_\E = d_{\CE}), \arrow[l,yshift = -0.7ex, "\Psi"]
	\end{tikzcd}
  \]
and thus an isomorphism on the cohomology level
\[
  H^\bullet(\Gamma(T_{A[1]}), L_{d_A}) \cong H^\bullet_{\CE}(A; E) = H^\bullet_{\CE}(A; K[1] \oplus B) = \bigoplus_n H^n_{\CE}(A; K[1] \oplus B),
\]
where $H^n_{\CE}(A; K[1] \oplus B):= H(\Omega^{n+1}_A(K[1]), d_{\CE}) \oplus H(\Omega^n_A(B), d_{\CE})$ is the Chevalley-Eilenberg cohomology of the graded $A$-module $K[1] \oplus B$ of \textbf{total degree} $n$.

By taking dual and tensor product, we obtain an isomorphism
\[
 H^\bullet(\Gamma(T^\ast _{A[1]} \otimes \End(T_{A[1]})), L_{d_A}) \cong H_{\CE}^\bullet(A; E^\ast \otimes \End(E)) = H_{\CE}^\bullet(A; (B^\ast \oplus K^\ast[-1]) \otimes \End(K[1] \oplus B)).
\]
In particular, we have
\begin{align*}
 H^1(\Gamma(T^\ast _{A[1]} \otimes \End(T_{A[1]})),& L_{d_A}) \cong H_{\CE}^1(A; (B^\ast \oplus K^\ast[-1]) \otimes \End(K[1] \oplus B)) \\
 &= H_{\CE}^1(A; K^\ast[-1] \otimes \End(K[1])) \oplus H_{\CE}^1(A; (K^\ast[-1] \oplus B^\ast) \otimes \Hom(B, K[1])) \\
 &\quad \oplus H^1_{\CE}(A; B^\ast \otimes \Hom(K[1], B)) \oplus H^1_{\CE}(A; B^\ast \otimes \End(B)).
\end{align*}
Under this isomorphism, we would like to write each class $w \in H^1(\Gamma(T^\ast _{A[1]} \otimes \End(T_{A[1]})), L_{d_A})$ in the following matrix form
\begin{equation}\label{Eq: matrix form for H1}
    \left(
      \begin{array}{cc}
        w_1 & w_2 \\
        w_3 & w_4 \\
      \end{array}
    \right),
\end{equation}
where
\begin{align*}
  w_1 &\in H_{\CE}^1(A; K^\ast[-1] \otimes \End(K[1])), & w_2 &\in H_{\CE}^1(A; (B^\ast \oplus K^\ast[-1]) \otimes \Hom(B,K[1])), \\
  w_3 &\in H^1_{\CE}(A; B^\ast \otimes \Hom(K[1],B)), & w_4 &\in H^1_{\CE}(A; B^\ast \otimes \End(B)).
\end{align*}
Meanwhile, in this case, the isomorphism~\eqref{Eq: isom for kforms} of cohomology of $k$-forms on the dg manifold $(A[1],d_A)$ becomes
\begin{equation}\label{Eq: Isom for kforms in locally splittable case}
 H^k(\Omega^k(A[1]), L_{d_A}) \cong \bigoplus_{q=0}^k H^k_{\CE}(A; \wedge^{k-q} B^\ast \otimes (S^q K^\ast)[-q]).
\end{equation}
\begin{lemma}\label{Lem: locally split}
    Let $A$ be a locally splittable regular Lie algebroid over $M$. Then the kernel $K$ of its anchor $\rho_A$ is a Lie algebra bundle, that is, the fiber Lie algebra is fixed in local trivialization of $K$.
\end{lemma}
\begin{proof}
According to Proposition 2.13 in~\cite{AC}, it suffices to show that there exists a linear connection on $K$ such that its basic curvature vanishes.
By assumption, one can choose a quadruple $(\tau, j, \nabla^K, \nabla^F)$ as in~\eqref{Eq: quadruple} such that $\Omega = 0$.
For all $a_K, a_K^\prime \in \Gamma(K), b \in \Gamma(B)$, and $u_F \in \Gamma(F)$,  we have
\begin{align*}
  R^{\bas}_{\nabla^K}(a_K, a_K^\prime)(u_F) &= \nabla^K_{u_F}[a_K, a_K^\prime]_K - [\nabla^K_{u_F}a_K, a_K^\prime]_K - [a_K, \nabla^K_{u_F}a_K^\prime]_K  \\
  &= [\tau(u_F), [a_K, a_K^\prime]_A]_A - [[\tau(u_F), a_K]_A, a^\prime_K]_A - [a_K, [\tau(u_F), a_K^\prime]_A]_A = 0,
\end{align*}
and
\[
 R^{\bas}_{\nabla^K}(a_K, a_K^\prime)j(b) = -\Omega(b)(a_K, a_K^\prime) = 0.
\]
Thus, the basic curvature $R^{\bas}_{\nabla^K}$ of the linear connection $\nabla^K$ on $K$ vanishes.
\end{proof}
As a consequence, the isomorphism~\eqref{Eq: isom of TK1} induced from the chosen linear connection $\nabla^K$ on $K$ becomes
\[
   (\Gamma(T_{K[1]}), L_{d_K}) \xrightarrow{\;\cong\;} (\Omega_K(K[1] \oplus T_M), d_{\CE}^K).
\]
By taking dual and tensor products on cohomology spaces, we obtain isomorphisms
\begin{equation}\label{Eq: Isom of 21 fields of K}
H^1(\Gamma(T_{K[1]}^\ast \otimes \End(T_{K[1]}))) \cong H^1_{\CE}(K; (T^\ast_M \oplus K^\ast[-1]) \otimes \End(K[1] \oplus T_M)),
\end{equation}
and
\begin{equation}\label{Eq: Isom for kform on K}
 H^k(\Gamma(T_{K[1]}), L_{d_K}) \cong \bigoplus_{q=0}^k H^k_{\CE}(K; \wedge^{k-q} T^\ast M \otimes (S^{q}K^\ast)[-q]) = \bigoplus_{q=0}^k  \Omega^{k-q}(M) \otimes_{C^\infty(M)} H^0_{\CE}(K; S^{q}K^\ast)[-q].
\end{equation}

\begin{theorem}\label{Thm B}
  Let $A$ be a locally splittable regular Lie algebroid, i.e., the characteristic class
  \[
  [\omega] \in H^2_{\CE}(A; \Hom(B,K))
  \]
  vanishes.
    \begin{enumerate}
  \item The Atiyah class $\At_{(A[1],d_A)} \in H^1(\Gamma(T_{A[1]}^\ast \otimes \End(T_{A[1]}))) \cong H_{\CE}^1(A; E^\ast \otimes \End(E))$ of the dg manifold $(A[1], d_A)$, when written in the matrix form as in~\eqref{Eq: matrix form for H1}, is of the block-diagonal type
  \[
    \At_{(A[1],d_A)} = \left(
                         \begin{array}{cc}
                           [\alpha_A] & 0 \\
                           0 & \rho_A^\ast\At_B \\
                         \end{array}
                       \right),
  \]
where
\begin{itemize}
\item  $[\alpha_A] \in  H^1_{\CE}(A, K^\ast[-1] \otimes \End(K[1]))$  is represented by the Lie bracket $[-,-]_K$ on $\Gamma(K)$, satisfying
 \begin{align*}
 i^\ast [\alpha_A] = \At_{(K[1],d_K)} &\in H^1_{\CE}(K, K^\ast[-1] \otimes \End(K[1])) \\
 &\subset H^1_{\CE}(K; (T^\ast_M \oplus K^\ast[-1]) \otimes \End(K[1] \oplus T_M))
 \end{align*}
  under the isomorphism~\eqref{Eq: Isom of 21 fields of K};
 \item $\rho_A^\ast\At_B \in H^1_{\CE}(A, B^\ast \otimes \End(B))$ is the pullback of the Atiyah class $\At_B \in H^1_{\CE}(F, B^\ast \otimes \End(B))$ of the Lie pair $(T_M, F)$ via the anchor map $\rho_A$.
 \end{itemize}
  \item Under the isomorphisms~\eqref{Eq: Isom for kforms in locally splittable case} and~\eqref{Eq: Isom for kform on K}, the $k$-th scalar Atiyah class $\ch_k(A[1],d_A)$  is given by the sum
  \[
   \ch_k(A[1],d_A) = \ch_k(K[1],d_K) + \rho_A^\ast(\ch_k(B)),
   \]
    where
    \begin{itemize}
   \item $\ch_k(K[1],d_K) \in H^0_{\CE}(K, S^kK^\ast)[-k]$ is the $k$-th scalar Atiyah class of $(K[1],d_K)$, which is also an element in $H^0_{\CE}(A, S^kK^\ast)[-k]$; and
    \item $\rho_A^\ast(\ch_k(B)) \in H^k_{\CE}(A; \wedge^kB^\ast)$ is the pullback (by the anchor map $\rho_A$) of the $k$-th scalar Atiyah class $\ch_k(B) \in H^k_{\CE}(F; \wedge^kB^\ast)$ of the Lie pair $(T_M, F)$.
    \end{itemize}
  \item Under the isomorphism
  \[
    \prod_{k \geq 0} H^k(\Omega^k(A[1]), L_{d_A}) \cong \prod_{k\geq 0} \bigoplus_{q=0}^k H^k_{\CE}(A; \wedge^{k-q} B^\ast \otimes (S^q K^\ast)[-q]).
  \]
  induced by~\eqref{Eq: Isom for kforms in locally splittable case}, the Todd class $\Td_{(A[1],d_A)} \in \prod_{k\geq 0} H^k(\Omega^k(A[1]), L_{d_A})$ is given by
  \[
       \Td_{(A[1],d_A)} = \Td_{(K[1],d_K)} \cdot \rho_A^\ast\Td_{B},
  \]
  where
  \begin{itemize}
  \item $\Td_{(K[1],d_K)} \in \prod_{k\geq 0} H^0_{\CE}(A; S^{k}K^\ast)[-k] \subset \prod_{k\geq 0} H^0_{\CE}(K; S^{k}K^\ast)[-k]$ is the Todd class of the dg manifold $(K[1], d_K)$, which is represented by the Duflo element of the Lie algebra bundle $K$;
  \item $\rho_A^\ast(\Td_B) \in \oplus_{k \geq 0} H^k_{\CE}(A; \wedge^k B^\ast) \subset \prod_{k\geq 0} \bigoplus_{q=0}^k  \Omega^{k-q}(M) \otimes_{C^\infty(M)} H^0_{\CE}(K; S^{q}K^\ast)[-q]$ is the pullback of the Todd class $\Td_B \in \oplus_{k \geq 0} H^k_{\CE}(F; \wedge^k B^\ast)$ of the Lie pair $(T_M, F)$.
  \end{itemize}
\end{enumerate}
\end{theorem}
\begin{proof}
By Lemma~\ref{Lem: locally split} and Corollary~\ref{corollary on Lie algebra bundle}, under the isomorphism~\eqref{Eq: Isom of 21 fields of K},
the Atiyah class $\At_{(K[1],d_K)}$ of the bundle $K = \ker \rho_A$ of Lie algebras arising from a locally splittable regular Lie algebroid $(A,\rho_A, [-,-]_A)$ lives in
\[
  H^1_{\CE}(K; K^\ast[-1]\otimes \End(K[1])) \subset H^1_{\CE}(K; (T^\ast_M \oplus K^\ast[-1]) \otimes \End(K[1] \oplus T_M)),
\]
and is represented by the Lie bracket $[-,-]_K$ on $\Gamma(K)$.

Applying Theorem~\ref{Theorem on Atiyah class of A[1]} to the locally splittable regular Lie algebroid $A$, we see that the Atiyah class of $(A[1], d_A)$
\[
 \At_{(A[1],d_A)} \cong \At_{(\E,Q_\E)} \in H^1_{\CE}(A; (K^\ast[-1] \oplus B^\ast) \otimes \End(K[1] \oplus B))
\]
when written in the matrix form as in~\eqref{Eq: matrix form for H1}, is represented by
   \[
    \left(
    \begin{array}{cc}
     \alpha_A & 0 \\
       0 & \alpha_B \\
      \end{array}
       \right),
   \]
where
\begin{itemize}
\item $\alpha_A$ is the Lie bracket $[-,-]_K$ on $\Gamma(K)$, satisfying $i^\ast [\alpha_A] = \At_{(K[1], d_K)}$;
\item $\alpha_B \in \Omega_A^1(B^\ast \otimes \End(B))$ coincides with the pullback $\rho_A^\ast \At^{\nabla^B}_B$ of the Atiyah cocycle $\At^{\nabla^B}_B$ of the Lie pair $(T_M, F)$.
\end{itemize}
This proves the first statement.

For the statement $(2)$, by Proposition~\ref{prop on scaler atiyah classes}, we have
\begin{align*}
   \ch_k(A[1],d_A) &= \frac{1}{k!}\left(\frac{i}{2\pi}\right)^k\left(\tr([\alpha_B]^k) - \tr([\alpha_A]^k)\right) \\
   &=  \frac{1}{k!}\left(\frac{i}{2\pi}\right)^k\tr\left([\alpha_B]^k\right) + \ch_k(K[1],d_K) \quad (\text{by the definition of $\alpha_B$ and $\At_B$})\\
   &=  \frac{1}{k!}\left(\frac{i}{2\pi}\right)^k\tr\left((\rho_A^\ast\At_B)^k\right) + \ch_k(K[1],d_K) \\
   &= \rho_A^\ast \ch_k(B) + \ch_k(K[1],d_K).
\end{align*}
For the statement $(3)$, we apply Proposition~\ref{prop on Todd class} to obtain
\begin{align*}
 \Td_{(A[1],d_A)} &= \det(P^{-1}([\alpha_A]))\cdot \det(P([\alpha_B])) = \det(P^{-1}([\alpha_A])) \cdot \det(P(\rho_A^\ast\At_B)) \\
 &= \det(P^{-1}([\alpha_A])) \cdot \rho_A^\ast(\det(P(\At_B))) \\
 &= \Td_{(K[1],d_K)} \cdot \rho_A^\ast(\Td_B).
\end{align*}
\end{proof}

\end{document}